\documentclass[a4paper,10pt]{article}%

\usepackage[utf8]{inputenc}%
\usepackage[T1]{fontenc}%
\usepackage[english]{babel}%

\usepackage{lmodern}%
\usepackage[final]{graphicx}%
\usepackage{xcolor}%
\usepackage{amsmath,amssymb,amsthm,dsfont}
\usepackage{hyperref}%
\usepackage{enumerate}%
\usepackage[labelfont=bf,labelsep=period]{caption}%

\allowdisplaybreaks%
\newcommand*{\Nbb}{\mathbb{N}}
\newcommand*{\Rbb}{\mathbb{R}}
\newcommand*{\Sbb}{\mathbb{S}}
\newcommand*{\Zbb}{\mathbb{Z}}

\newcommand*{\gb}{\mathbf{g}}%
\newcommand*{\hb}{\mathbf{h}}%

\newcommand*{\Gmb}{{\boldsymbol\Gamma}}%
\newcommand*{\veps}{\varepsilon}%
\newcommand*{\etab}{{\boldsymbol\eta}}%
\newcommand*{\vphi}{\varphi}%
\newcommand*{\kpb}{{\boldsymbol\kappa}}%
\newcommand*{\la}{\lambda}%
\newcommand*{\mub}{{\boldsymbol\mu}}%
\newcommand*{\pib}{{\boldsymbol\pi}}%
\newcommand*{\sigb}{{\boldsymbol\sigma}}%
\newcommand*{\vrho}{\varrho}%
\newcommand*{\om}{\omega}%
\newcommand*{\omb}{{\boldsymbol\omega}}%

\newcommand*{\Dc}{\mathcal{D}}%
\newcommand*{\Gc}{\mathcal{G}}%
\newcommand*{\Ic}{\mathcal{I}}%
\newcommand*{\Lc}{\mathcal{L}}%
\newcommand*{\Nc}{\mathcal{N}}%
\newcommand*{\Rc}{\mathcal{R}}%

\newcommand*{\oh}{\frac12}%
\newcommand*{\thh}{\frac32}%
\newcommand*{\di}{\,\mathrm{d}}%

\newcommand*{\emp}[1]{\emph{#1}}%
\newcommand*{\abs}[1]{\left\lvert #1 \right\rvert}%
\newcommand*{\omin}{\om^{\mathrm{min}}}%
\newcommand*{\omax}{\om^{\mathrm{max}}}%

\numberwithin{equation}{section}%

\newtheorem{assumption}{Assumption}%

\theoremstyle{plain}%
\newtheorem{theorem}{Theorem}%
\newtheorem{proposition}{Proposition}%

\theoremstyle{definition}%
\newtheorem{definition}{Definition}%
\newtheorem{remark}{Remark}%

\newcommand{\qqed}{\hfill$\blacksquare$}%

\renewenvironment*{proof}[1][Proof]{%
\begin{trivlist}%
\item[\hskip \labelsep {\bfseries #1}]}{\qqed\end{trivlist}}%

\newcommand{\Exp}{\operatorname{\mathbf{E}}}%
\newcommand{\Prob}{\operatorname{\mathbf{P}}}%
\newcommand{\Var}{\operatorname{\mathbf{Var}}}%
\newcommand{\ind}{\operatorname{\mathds{1}}}%

\newcommand*{\bristol}{School of Mathematics, University of Bristol, University Walk, Bristol, BS8 1TW, United Kingdom.}%
\newcommand*{\bute}{Institute of Mathematics, Budapest University of Technology and Economics, Egry J. u. 1., Budapest, H-1111, Hungary.}

\newcommand*{\titl}{
How to initialize a second class particle?
}%

\newcommand*{\auths}{
M\'arton Bal\'azs\,\footnote{\bristol}\\[0.1em]
{\normalsize \url{m.balazs@bristol.ac.uk}}
\and
Attila L\'aszl\'o Nagy\,\footnote{\bute}\\[0.1em]
{\normalsize \url{attilalaszlo.nagy@gmail.com}}
}%

\newcommand*{\dat}{\today}%

\title{\titl}%
\author{\auths}%
\date{\dat}%

\begin{document}
\maketitle%

\begin{abstract}
 We identify the ballistically and diffusively rescaled limit distribution of the second class particle position in a wide range of asymmetric and symmetric interacting particle systems with established hydrodynamic behavior, respectively (including zero-range, misanthrope and many other models). The initial condition is a step profile which, in some classical cases of asymmetric models, gives rise to a rarefaction fan scenario. We also point out a model with non-concave, non-convex hydrodynamics, where the rescaled second class particle distribution has both continuous and discrete counterparts. The results follow from a substantial generalization of P.\ A.\ Ferrari and C.\ Kipnis' arguments (``Second class particles in the rarefaction fan'', Ann. Inst. H. Poincar\'e, 31, 1995) for the totally asymmetric simple exclusion process. The main novelty is the introduction of a \emp{signed} coupling measure as initial data, which nevertheless results in a proper probability initial distribution for the site of the second class particle and makes the extension possible. We also reveal in full generality a very interesting invariance property of the one-site marginal distribution of the process underneath the second class particle which in particular proves the intrinsicality of our choice for the initial distribution. Finally, we give a lower estimate on the probability of survival of a second class particle-antiparticle pair.
\end{abstract}


\bigskip
\noindent \textbf{Acknowledgement.} The authors thank valuable discussions with Pablo A.\ Ferrari on the problem, with Ellen Saada on the hydrodynamic limit of asymmetric processes and with C\'edric Bernardin on symmetric processes. We also thank the anonymous referee for his/her comments. M.\ Bal\'azs acknowledges support from the Hungarian National Research, Development and Innovation Office, NKFIH grants K100473 and K109684.



\section{Introduction}%

This paper studies the behavior of second class particles in a wide class of one-dimensional attractive particle systems. The evolution of such particles can be obtained by coupling two systems (of first class particles) coordinate-wise in such a manner that their initial configurations only differ at finitely many places. Second class particles interact with the underlying process and perform highly nontrivial motion which is only partially understood in general. In asymmetric models they are known, in first order, to follow the characteristic lines of the limiting hydrodynamic equation of the density. In three classical cases: translation-invariant stationary, rarefaction fan, and shock scenario this results in a law of large numbers with the characteristic velocity, a random admissible characteristic velocity of the rarefaction fan, and the speed of the shock, respectively. These make the second class particle a relevant microscopic object that captures macroscopic properties of the ambient system. Fluctuations show superdiffusive scaling for translation-invariant stationary, rarefaction fan, deterministic shock initial data and diffusive scaling for random shock initial data. Many of the previous properties have been proven rigorously for the most-studied totally asymmetric simple exclusion process (TASEP) and in some cases for other processes as well. However, they are conjectured to hold in a wide range of particle systems. Second class particles in symmetric systems have not been much explored, in some simple cases diffusive behavior is known.

We build on the seminal paper \cite{serf} by P.\ A.\ Ferrari and C.\ Kipnis which made use of a translation argument to investigate the second class particle of the TASEP starting from the rarefaction fan. Their argument compares a step initial product Bernoulli distribution with its translated version and notices that the joint realization of these two can be understood as a coupled initial distribution with possibly a second class particle at the origin. This program crucially relied on the fact that the second class particle of simple exclusion is a uniquely defined object as it can \emp{only} conceive by coupling a process of \emp{zero} particles with one of \emp{one} particle at the site of the second class particle. When dealing with systems of more choices for one site occupation numbers, the second class particle stops being a uniquely determined object. Stochastic domination of the natural measures associated with attractive models still holds but the actual realization of a coupled pair has some details to fix besides its marginals. In particular, it is not clear whether two models with slightly different densities can be coupled using zero or one second class particles per site only, or more than one of them on a site have to be assumed with positive probability. Actually, this latter is the case for popular stationary distributions as the ones of Geometric or Poisson marginals (e.g., for zero-range processes).

First, we build up a natural initial distribution ($\hat\mub^{\vrho,\,\la}$) for the second class particle in step initial configurations (with different densities $\vrho\neq\la$ on the left and right) which allows for an extension of P.\ A.\ Ferrari and C.\ Kipnis' arguments. Our construction works even when coupling with zero or one second class particles only fails. This is where the main novelty of the paper lies: to force zero or one second class particles with the correct one-site marginals for the coupled pair, one has to introduce negative weights in the coupling measure. As it turns out negative weights only belong to configurations without a second class particle, and this non-physical coupling measure always assigns positive weights to states with a second class particle. By normalizing on these states only, the proper probability distribution $\hat\mub^{\vrho,\,\la}$ a.s.\ has then the second class particle, which will also turn out to be canonical in many sense.

Under the initial distribution $\hat\mub^{\vrho,\,\la}$ we connect the displacement of the second class particle to easier quantities of the ambient system. Using recent results of hydrodynamics we can then proceed to prove limit distribution results on the rescaled position of the second class particle. Both asymmetric and symmetric systems are handled under the natural scaling that fits the respective scenario. The limit distributions then relate to the solution of the hydrodynamic equation with step initial condition. There are two particular and interesting instances, to the best of our knowledge not much explored in the literature, of second class particle-behavior:
\begin{enumerate}[(i)]
\item in asymmetric models with non-concave and non-convex hydrodynamic flux, shocks and rarefaction fans can coexist and the limit distribution of the second class particle reflects this fact by developing both continuous and discrete components at the same time; and
\item central limit theorem for the second class particle is pointed out in a symmetric system where, as opposed to simple symmetric exclusion, it is \emp{not} a simple random walker.
\end{enumerate}

As a by-product of our arguments we are able to relate the one-site marginal of the first class particles at the site of the second class particle to the distribution of a model without the second class particle. Under certain initial distributions this results in a time-stationary one-site marginal -- a quite unexpected result. Finally, we push the arguments, in line with \cite{serf}, to give a lower estimate on the survival probability of a second class particle-antiparticle pair in general models.

\bigskip
\noindent\emp{Earlier results.}
A review and several open problems appeared in \cite{ferrarishocks,guiolmountfordquestions} many of which are completely solved by the present paper. A law of large numbers for the position of the second class particle of exclusion and zero range processes with shock initial condition was obtained by F.\ Rezakhanlou \cite{rez}. Note that his initial setup of the second class particle slightly differs from ours. As described above, in case of the rarefaction fan (and for the TASEP) the first and fundamental paper was \cite{serf}. T.\ Mountford and H.\ Guiol \cite{mountguiol} then sharpened \cite{serf} by proving that the convergence takes place almost surely. Recently P.\ Gon\c{c}alves has translated the results of \cite{serf} for the totally asymmetric constant rate zero-range process in \cite{gon_zr_raref} via a direct coupling between exclusion and zero-range. P.\ A.\ Ferrari, P.\ Gon\c{c}alves and J.\ B.\ Martin \cite{fgmcollision} have very elegant arguments on collision probabilities in exclusion processes. Many results on the behavior of the second class particle in the TASEP have been reproven by P.\ A.\ Ferrari and L.\ P.\ R.\ Pimentel \cite{compint} and by P.\ A.\ Ferrari, J.\ B.\ Martin and L.\ P.\ R.\ Pimentel \cite{ferrmarpiment}, translating the problem into one of competition interfaces in last passage percolation.
E.\ Cator and S.\ Dobrynin \cite{catdob} have studied Hammersley's process in continuous space in which limit theorems were obtained for the second class particle starting from the rarefaction fan. D.\ Romik and P.\ \'{S}niady \cite{romik} pointed out an elegant algebraic connection between the motion of second class particles in a variant of the TASEP and an evolution, so-called ``jeu de taquin'', defined on infinite Young tableaux through which the distributional limit was proved. TASEP equipped with higher order particles (like third, fourth, etc.\ class particles), known as the \emp{multi-type} TASEP, was investigated in \cite{amirangel} by G.\ Amir, O.\ Angel and B.\ Valk\'o, where the joint distribution of the speeds of higher order particles were also identified. This in particular includes collision probabilities and the formation of convoys. Analytic formul\ae\ were obtained by C.\ A.\ Tracy and H.\ Widom \cite{tracywidom} for the second class particle starting from the rarefaction fan of the ASEP.

\bigskip
\noindent\emp{Organization of the paper.} We start with discussing initial distributions in Subsection \ref{sec:initmeas} which form a crucially important part of our arguments. We then proceed with describing the dynamics in Section \ref{sec:dyn} with additional requirements in Section \ref{sec:misa}. The second class particle, our main object, is introduced in Section \ref{sec:scp}. We early on state the main results of this paper in Section \ref{sec:mainr} for which the precise hydrodynamic statements we need are postponed to Section \ref{sec:hydro} due to organizational purposes. Remarks on the initial distribution (Section \ref{sec:thed}), the fundamental identity behind our results (Section \ref{sec:scpdistr}), and a theorem on the background distribution of the site of the second class particle (Section \ref{sec:bg}) are also slightly postponed. We outline and discuss several examples of models in Section \ref{sec:partmodels}. Proofs follow in Section \ref{sec:proof}.

\section{Models}%

\subsection{State space and initial distribution}\label{sec:initmeas}
The model class we investigate originates in the work of Cocozza-Thivent \cite{coco}, extensions and several examples we cover first appeared in the papers \cite{fluct,hydro}. We consider general, nearest neighbor stochastic interacting particle systems $\omb:\,=(\omb(t))_{t\geq0}=\big((\om_i(t))_{i\in\Zbb}\big)_{t\geq0}$ on the configuration space $\Omega:=\Ic^{\Zbb}$ with $\Ic=\{\omin,\,\omin+1,\,\ldots,\,\omax-1,\,\omax\}\subset\Zbb$ such that $-\infty\leq\omin<\omax\leq+\infty$. In particular $\Ic$ can as well be an infinite subset of $\Zbb$. The quantity $\om_i(t)$ denotes the number of (signed) first class particles sitting at the $i^{\mathrm{th}}$ lattice point at time $t\in\Rbb^+_0$. We adopt this interpretation even if $\om_i(t)$ happens to be negative.

Our main object of investigation is the \emp{second class particle} which comes up from couplings of systems of first class particles. In particular it lives in the space $\Omega\times\Omega$, so before describing the dynamics of the above systems the appropriate choice of the initial measure on $\Omega\times\Omega$ will be discussed. This measure to be defined later turns out to be canonical and is indeed one of the crucial points of this paper.

We start with a general assumption on one-site marginals which will be the basis of building product initial distributions of configurations in $\Omega$ and of coupled pairs of configurations in $\Omega\times\Omega$.
\begin{assumption}\label{assump:marginalmeasure}
Let $\nu:\,=\left(\nu^{\vrho}\right)_{\vrho\in \Dc}$ be a family of probability measures on $\Ic$,
where $\Dc$ is a bounded subset of $[\omin,\omax]$ that satisfies the following properties:
\begin{itemize}
\item it is parametrized by its mean, that is $\vrho = \sum_{y\in\Ic}\,y\cdot\nu^{\vrho}(\{y\})$ holds for every $\vrho\in\Dc$; and
\item for each $\vrho>\la$, where $\vrho,\la\in \Dc$, the measure $\nu^{\vrho}$ stochastically dominates $\nu^{\la}$, that is $\nu^{\la}(\{z:z\leq y\}) \geq \nu^{\vrho}(\{z:z\leq y\})$ holds for every $y\in\Ic$.
\end{itemize}
\end{assumption}
This assumption is very mild, for e.g.\ any deterministic marginals of the form $\nu^{\vrho}(x) = \ind\{x=\vrho\}$ satisfy it with a $\vrho\in\Dc:\,=[\omin,\omax]\cap\Zbb$. We will present a more general set of measures in Section \ref{sec:misa} that also satisfies the above assumption.

In the sequel, we will refer to $\Dc$ as the \emp{set of densities}. With $\vrho,\ \la\in\Dc$, we define the product distribution
\begin{equation}\label{eq:initprodmeas}
\sigb^{\vrho,\,\la} :\,= \bigotimes_{i=-\infty}^0\nu^{\vrho}\otimes\bigotimes_{i=1}^{+\infty}\nu^{\la}
\end{equation}
on $\Omega$. Whenever $\vrho\neq\la$ this will be called the microscopic \emp{Riemannian density profile} or simply the step initial condition.

Next, we turn to special distributions on $\Omega\times\Omega$. Fix two densities $\vrho>\la$ of $\Dc$ and we define the measure $\hat{\nu}^{\vrho,\,\la}$ on $\Ic\times\Ic$ as
\begin{equation}\label{eq:scpcouplatorig}
\hat{\nu}^{\vrho,\,\la}(x,y)=
\frac{1}{\vrho-\la}\big(\nu^{\la}(\{z:z\leq y\})-\nu^{\vrho}(\{z:z\leq y\})\big)\cdot\ind\{x=y+1\},
\end{equation}
where $x,y\in\Ic$. It is an easy exercise to check that this indeed defines a probability distribution. We will comment on its origin later in Section \ref{sec:thed}.
Notice that $\om_0=\eta_0+1$ holds $\hat{\nu}^{\vrho,\,\la}$-a.s.\ for its two marginals. By a slight abuse of notation we also set
\begin{equation}\label{eq:diagonalmeasures}
\nu^{\vrho,\,\vrho}(x,\,y):\,=\nu^\vrho(x)\cdot\ind\{x=y\},
\quad\text{ and }\quad
\nu^{\la,\,\la}(x,\,y):\,=\nu^\la(x)\cdot\ind\{x=y\}
\end{equation}
as diagonal measures on $\Ic\times\Ic$. We can now define the initial probability distribution as a site-wise product coupling measure on the space $\Omega\times\Omega$:
\begin{equation}\label{eq:scpstart}
\hat\mub^{\vrho,\,\la}:\,=
\bigotimes_{i=-\infty}^{-1}\nu^{\vrho,\,\vrho}\otimes\hat\nu^{\vrho,\,\la}\otimes\bigotimes_{i=1}^\infty\nu^{\la,\,\la}.
\end{equation}
Later, we will start a coupled pair of systems of first class particles under the initial distribution $\hat{\mub}^{\vrho,\,\la}$, and we denote the associated probability and expectation by $\hat\Prob$ and $\hat\Exp$, respectively. Though the precise notion of the second class particle will be defined in Subsection \ref{sec:scp}, here we notice in advance that $\hat\Prob$ a.s.\ has a second class particle that initially starts from the origin.

\subsection{Dynamics of the models}\label{sec:dyn}
A continuous time Markov jump dynamics is attached on top of the configuration space $\Omega$ that allows the particles to execute right as well as left jumps with respective instantaneous rates $p$ and $q$. Formally, with the \emp{Kronecker symbol} $(\delta_i)_j=\ind\{i=j\}$ ($\ind\{\,\cdot\,\}$ stands for the indicator function throughout the article), the transitions are of the form
\begin{equation}
\omb\;\xrightarrow{\displaystyle p(\om_i,\om_{i+1})}\;\omb-\delta_{i}+\delta_{i+1}\in\Omega;\qquad
\omb\;\xrightarrow{\displaystyle q(\om_i,\om_{i+1})}\;\omb+\delta_{i}-\delta_{i+1}\in\Omega,
\label{eq:dynwpq}
\end{equation}
where $p,q:\Ic\times\Ic\to\Rbb^+_0$ are given \emp{deterministic} functions. Conditioned on a given configuration, the above steps take place independently for each $i\in\Zbb$ with the above respective rates. Throughout the article we assume \emp{non-degeneracy} for the rates, that is for every $i\in\Zbb$: $p(\om_i,\om_{i+1})>0$ ($q(\om_i,\om_{i+1})>0$) \emp{if and only if} $\omin<\om_i$ and $\om_{i+1}<\omax$ ($\omin<\om_{i+1}$ and $\om_i<\omax$). This also makes sure that the process a.s.\ keeps the state space $\Omega$. Sometimes we will let one of the left or right jump rates be zero (totally asymmetric case).

Now, the (formal) \emp{infinitesimal generator} $\Gc$ of our Markov process acts on a \emp{cylinder function} $\vphi:\Omega\to\Rbb$ (one that depends only on a finite number of coordinates of $\omb\in\Omega$) as
\begin{equation}\label{eq:infgen}
\begin{aligned}
\big(\Gc\,\vphi\big)(\omb)
=&\sum_{j\in\Zbb}\,p(\om_{j},\om_{j+1})\cdot\big(\vphi(\omb-\delta_j+\delta_{j+1})-\vphi(\omb)\big)\\
+&\sum_{j\in\Zbb}\,q(\om_{j},\om_{j+1})\cdot\big(\vphi(\omb+\delta_j-\delta_{j+1})-\vphi(\omb)\big).
\end{aligned}
\end{equation}
If $p$ and $q$ are bounded functions on $\Ic^2$ then the above Markov process can be constructed on $\Omega$ in an appropriate manner having generator $\Gc$ (see \cite[Chapter 1]{ips}). In other cases, existence of the dynamics can only be established by posing further (growth) conditions on the rates (see \cite{and}, \cite{exists} and further references therein). Within the scope of this article we do not intend to deal with this issue in general, though we will discuss some models with unbounded rates in Section \ref{sec:partmodels}. From now on we assume that the processes can be constructed with appropriate initial data in $\Omega$ with the above dynamics. In the next subsection we introduce the \emp{attractiveness assumption} which will further tighten the model class.

\subsection{Second class particles}\label{sec:scp}
Pick two configurations $\omb$ and $\etab$ both in $\Omega$ aligning them coordinate-wise. Then one can define the \emp{number} $n_i=\abs{\om_i-\eta_i}$ and the \emp{sign} $s_i=\ind\{\om_i-\eta_i>0\}-\ind\{\om_i-\eta_i<0\}$ of \emp{signed second class particles} at position $i\in\Zbb$ in the configuration pair $(\omb,\,\etab)$. In particular, if
\begin{equation}\label{eq:lonescp}
\omb = \etab + \delta_0 \qquad (\omb = \etab - \delta_0),
\end{equation}
then we say that a single positive (negative) second class particle is placed at the origin in $(\omb,\etab)$. To allow second class particles evolve in time we use the \emp{basic, ``particle-to-particle'', coupling}, that is for each time $t>0$ and lattice point $i\in\Zbb$, a hop to the right can occur in both systems:
\begin{equation*}
\big(\omb,\etab\big) \longrightarrow \big(\omb-\delta_i+\delta_{i+1},\, \etab-\delta_i+\delta_{i+1}\big)\in\Omega\times\Omega
\end{equation*}
with rate $\min\big(p(\om_i,\om_{i+1}), p(\eta_i,\eta_{i+1})\big)$; while ``compensating'' right jumps occur according to the following rules with respective rates:
\begin{center}
\begin{tabular}{c|c|c}
& $\big(\omb-\delta_i+\delta_{i+1},\, \etab\big)$ & $\big(\omb,\, \etab-\delta_i+\delta_{i+1}\big)$\\[0.25em]
\hline\\[-1.2em]
$\big(\omb,\etab\big)$ & $\big(p(\om_i,\om_{i+1})-p(\eta_i,\eta_{i+1})\big)^+$ &
$\big(p(\om_i,\om_{i+1})- p(\eta_i,\eta_{i+1})\big)^-$
\end{tabular}\ .
\end{center}
Here $(\,\cdot\,)^+$ and $(\,\cdot\,)^-$ denote the positive and negative part function, respectively.
The coupling tables for the left jumps can be obtained analogously.
Note that a second class particle can hop only if a compensating step occurs. Also notice that under the basic coupling the marginal processes, that is $(\omb(t))_{t\geq0}$ and $(\etab(t))_{t\geq0}$, follow the same stochastic evolution rules \eqref{eq:dynwpq}. Now, recall the following notion from \cite[Definition 2.3, pp. 72]{ips}.
\begin{definition}[Attractiveness]
We say that the dynamics defined by the infinitesimal generator $\Gc$ of \eqref{eq:infgen} is \emp{attractive},
if the initial dominance $\etab(0)\leq \omb(0)$ implies the one $\etab(t)\leq \omb(t)$ for all times $t>0$ under the basic coupling.
\end{definition}
From now on we will always assume that the processes we consider are attractive. It is not hard to see that this is equivalent to saying that the rate $p$ ($q$) is monotone non-decreasing (non-increasing) in its first and monotone non-increasing (non-decreasing) in its second variable.

In attractive processes, the above basic coupling tables reveal some extra properties for the second class particles. In particular, having initial configurations as in \eqref{eq:lonescp}, a.s.\ there will always be a single second class particle in the system, the position of which will be denoted by $Q(t)$ at time $t$. More generally, one can see that the total number $\sum_{j\in\Zbb}\abs{\om_j(t)-\eta_j(t)}=\sum_{j\in\Zbb}n_j\,$ of (positive and negative) second class particles is non-increasing in time.

\section{Main results}\label{sec:mainr}

The results below heavily rely on the hydrodynamic description of particle systems. For the time being we skip the rather technical details of hydrodynamic limit theory. We refer to Subsections \ref{sec:asymhydro} and \ref{sec:symhydro} which are devoted to the precise definitions and statements on hydrodynamics where all the missing elements are fully expounded. Our first two results concern the limit distribution of the position of the second class particle.
\begin{theorem}[Speed of the second class particle in asymmetric models]\label{thm:scphydroasym}
Suppose Assumption \ref{assump:marginalmeasure}. Then start a second class particle at the origin from the product coupling measure $\hat\mu^{\vrho,\,\la}$ ($\vrho\neq\la$) (see \eqref{eq:scpstart}), where the underlying model of first class particles $\omb$ can \emp{either} be
\begin{itemize}
\item \emp{any attractive process} with bounded one site occupation numbers ($-\infty<\omin,\omax<+\infty$) and we have \emp{no} further assumptions on the measure $\nu$; \emp{or}
\item a \emp{misanthrope process} with \emp{bounded rate} functions (but not necessarily bounded occupations). In this case $\nu$ is restricted to be a stationary marginal.
\end{itemize}
Then we have the limit
\begin{equation}\label{eq:scpspeedlim}
\lim_{N\to\infty}\hat\Prob\bigg\{\frac{Q(Nt)}N\leq x\bigg\} = \frac{\vrho - u(x,t)}{\vrho - \la}
\end{equation}
for every $x\in\Rbb$ that is a continuity point of $u(\,\cdot\,,t)$, where $u\in\Rbb\times\Rbb_0^+$ is the unique entropy solution of the conservation law $\partial_t u + \partial_x G(u) = 0$ with step initial datum $u(0, x)=\vrho\ind\{x\leq 0\}+\la\ind\{x>0\}$ and hydrodynamic flux function $G$.
\end{theorem}
Next a couple of comments. First, we underline that for systems with bounded occupation numbers the limit \eqref{eq:scpspeedlim} holds for \emp{any} choice of marginal distribution $\nu$ satisfying Assumption \ref{assump:marginalmeasure}. On the other hand, we note that the misanthrope family of processes forms a large and important part of attractive particle systems. The rate functions of these satisfy further combinatorial identities which enable one to give a full description of the translation invariant stationary distributions. The corresponding results will be recapitulated in Section \ref{sec:misa} (see Theorem \ref{thm:prodstatergod}).

The hydrodynamic flux $G$ (which will be defined later in \eqref{eq:hydrofluxdef} in Section \ref{sec:hydro}) roughly speaking describes the average signed rate of jumping particles across a bond in stationarity. In some models strict concavity or convexity of $G$ has been established and it is then well understood that the Riemann (or step) initial condition ($\vrho\ind\{x\leq 0\}+\la\ind\{x>0\}$) develops shock or rarefaction fan solutions depending on the order of $\vrho$ and $\la$ and on concavity or convexity of $G$. In a shock, the limiting probability \eqref{eq:scpspeedlim} is of $0$-$1$ form which means convergence of the scaled second class particle position to the deterministic velocity of the shock. In a rarefaction fan we have convergence to a random velocity. This randomness is uniform for the totally asymmetric simple exclusion process, as $u$ is a linear function of its first -- spatial -- argument which has been first observed by P.\ A.\ Ferrari and C.\ Kipnis \cite{serf} but this distribution might vary with other models. We highlight that there are attractive models with product-form stationary distributions but with non-concave, non-convex hydrodynamic fluxes. In the associated conservation laws coexistence of shocks and rarefaction fans is possible, in which cases our result shows that the limit distribution of the velocity of the second class particle is \emp{mixed} with a discrete mass and a continuous counterpart (see for e.g.\ the $2$-type model of Section \ref{sec:partmodels}).

Our arguments are general enough to include \emph{symmetric models} as well which have interesting consequences for the second class particle in this case. The analogous result follows, and we will illustrate its significance with the symmetric zero range processes at the end of Section \ref{sec:partmodels}. See the definition of a gradient process and the related quantity $d$ in Section \ref{sec:symhydro}.
\begin{theorem}[Speed of the second class particle in symmetric models]\label{thm:scpdiffusivelim}
Suppose $-\infty<\omin$, Assumption \ref{assump:marginalmeasure} and let $\omb$ be a symmetric gradient process which is attractive. Then we have the limit
\begin{equation}\label{eq:scpdiffusivelim}
\lim_{N\to\infty}\hat\Prob\bigg\{\frac{Q(Nt)}{\sqrt{N}}\leq x\bigg\} = \frac{\vrho - u(x,t)}{\vrho - \la}
\end{equation}
for every continuity point $x\in\Rbb$ of $u(\,\cdot\,,t)$ provided that $\Exp_{\sigb^{\vrho,\,\vrho}}\omb_0(t)^2<+\infty$ holds for every $t\geq0$, where $u\in\Rbb\times\Rbb_0^+$ is the unique weak solution to the parabolic partial differential equation $\partial_t u=\text{{\small$\frac{1}{2}$}}\, \Delta d(u)$ with step initial datum $u(0, x)=\vrho\ind\{x\leq 0\}+\la\ind\{x>0\}$ and diffusivity coefficient $d$.
\end{theorem}
Symmetric gradient processes and their hydrodynamic properties (diffusivity) will be rigorously discussed in Subsection \ref{sec:symhydro}.

Finally, we focus on the interaction of two second class particles of opposite charges dropped into the system initially. Denote by $\Nc(t)$ the total number of second class particles present in the system at time $t$. For the long-time behavior of $\Nc$ we have the following result.
\begin{theorem}[Collision probability of second class particles]\label{thm:collprob}
Assume that $\omin$ and $\omax$ are finite numbers. Let $(\hat\omb,\hat\etab)$ be any pair of attractive systems starting from the deterministic initial configurations
\begin{equation*}
\hat\omb_0 = \hat\etab_0 -\delta_0 + \delta_1,\qquad \hat\etab_0 = \omax\ind\{i\leq 0\} + \omin\ind\{i>0\}
\end{equation*}
and evolving according to the basic coupling. Then
\begin{equation}\label{eq:collprobestimate}
\hat\Prob\big\{\Nc(t)=2\text{ for all }t\geq0\big\}\geq
\frac{\bar{G}(1)}{p(\omax, \omin)} =\,: C_0,
\end{equation}
where $\bar{G}(1) = \limsup_{N\to+\infty}\Exp_{\hat\etab_0}[p(\hat{\eta}_0(N),\hat{\eta}_1(N))-q(\hat{\eta}_0(N),\hat{\eta}_1(N))]$, while $\hat\Prob$ denotes the associated probability of $(\hat\omb(t),\hat\etab(t))_{t\geq0}$.

In particular, if the dynamics is totally asymmetric ($q\equiv0$) then $C_0>0$ holds. On the other hand, considering one of the misanthrope processes (described by Theorem \ref{thm:prodstatergod}) we have $\bar{G}(1) = G(u(0,1))$ provided that $0$ is a continuity point of $u(\,\cdot\,,1)$, where $G$ is the hydrodynamic flux (defined in \eqref{eq:hydrofluxdef}) and $u$ is the unique entropy solution to $\partial_t u + \partial_x G(u) = 0$ with step initial datum ($\omax\ind\{x\leq 0\} + \omin\ind\{x>0\}$).
\end{theorem}
The previous assertion tells that two second class particles of distinct charges initially placed at lattice points $0$ and $1$ will \emp{never} meet with positive probability provided that the constant $C_0$ of \eqref{eq:collprobestimate} is positive. For the asymmetric simple exclusion process with $\Ic=\{0,1\}$ and rate functions $p(\om_i,\om_{i+1})=\bar{p}\cdot\om_i\cdot(1-\om_{i+1})$ and $q(\om_i,\om_{i+1})=(1-\bar{p})\cdot \om_{i+1}\cdot(1-\om_{i})$ $(i\in\Zbb)$, where $\bar{p}\in(\oh,1]$, we recover the result \cite[Theorem 2]{serf}, if $\bar{p}=1$.
Indeed, we know exactly from \cite[Theorem 2.3]{fgmcollision} that for each $\bar{p}\in(\oh,1]$
\begin{equation}\label{eq:collprob}
\hat\Prob\big\{\Nc(t)=2\text{ for all }t\geq0\big\} = \frac{2\bar{p}-1}{3\bar{p}}
\end{equation}
for which \eqref{eq:collprobestimate}, $C_0$ being $\frac{2\bar{p}-1}{4\bar{p}}$, gives a \emp{non-sharp} lower bound. Formula \eqref{eq:collprob} was also derived from a more general model, known as the multi-type (T)ASEP speed process, in \cite[Theorem 1.12]{amirangel}.

\section{Additional results}\label{sec:furth}

In this section we state additional results, following from very general coupling arguments, that give further insight to phenomena under the initial distribution \eqref{eq:scpstart}. We first indicate where this initial distribution comes from. Then an intermediate step towards main Theorems \ref{thm:scphydroasym} and \ref{thm:scpdiffusivelim}, without any reference to hydrodynamics, is shown. Finally, we proceed with an invariance property of the model at the site of the second class particle.

\subsection{The distribution $\hat{\nu}^{\vrho,\,\la}$}\label{sec:thed}
The following will demonstrate why the measure \eqref{eq:scpstart} serves as a natural choice for initial distribution.
\begin{proposition}\label{prop:probcouplingexist}
Suppose that Assumption \ref{assump:marginalmeasure} holds and let $\vrho-1\leq\la<\vrho$, where $\vrho,\la\in\Dc$ are fixed. Then there exists a joint \emp{probability} measure $\nu^{\vrho,\,\la}$ with $\nu^\vrho$ and $\nu^\la$ as respective marginals and with $\nu^{\vrho,\,\la}(\{(x,y):x-y\in\{0,1\}\})=1$ \emp{if and only if}
\begin{equation}\label{eq:couplingstrict}
\nu^{\vrho}(\{z:z\leq y\})\geq \nu^{\la}(\{z:z\leq y-1\})
\end{equation}
holds for every $y\in\Ic$. In this case $\hat{\nu}^{\vrho,\,\la}$ can be obtained as
\begin{equation}\label{eq:condmeasure}
\hat\nu^{\vrho,\,\la}(\,\cdot\,)=
\nu^{\vrho,\,\la}\big(\,\cdot\,\big|\,\om_0=\eta_0+1\big)=
\frac{\nu^{\vrho,\,\la}(\,\cdot\,\cap\{(x,y):x=y+1\})}{\nu^{\vrho,\,\la}(\{(x,y):x=y+1\})},
\end{equation}
where $0<\nu^{\vrho,\,\la}(\{(x,y):x=y+1\})=\vrho-\la\leq 1$.
\end{proposition}
Under the narrower assumptions of Proposition \ref{prop:probcouplingexist}, we can set up another measure, namely
\begin{equation*}
\mub^{\vrho,\,\la}:\,=
\bigotimes_{i=-\infty}^{-1}\nu^{\vrho,\,\vrho}\otimes\nu^{\vrho,\,\la}\otimes\bigotimes_{i=1}^\infty\nu^{\la,\,\la},
\end{equation*}
which we can call the \emp{unconditional version} of $\hat{\mub}^{\vrho,\,\la}$, since this latter can be obtained from $\mub^{\vrho,\,\la}$ by conditioning on the existence of a single second class particle at the origin.

Some, but not all, interacting particle systems have translation-invariant product stationary distributions. For those with product measures, it seems natural to choose the marginals $\nu^\vrho$ and $\nu^\la$ to be these stationary marginals. As two classical examples, the product of Geometric and Poisson distributions on $\Zbb^{\Zbb_0^+}$ are stationary for zero-range processes with constant and linear rate functions, respectively, to be discussed in Section \ref{sec:partmodels} in more details.
Notice, as the following Proposition \ref{prop:negativeresult} also demonstrates, that the additional requirement \eqref{eq:couplingstrict} of Proposition \ref{prop:probcouplingexist} might be too restrictive in some cases where $\nu^{\vrho,\,\la}$, hence $\mub^{\vrho,\,\la}$, might \emp{not} exist as a \emp{probability} measure.
\begin{proposition}\label{prop:negativeresult}
The family of \emp{Geometric} as well as \emp{Poisson distributions} can be parametrized to fulfill Assumption \ref{assump:marginalmeasure} but there do \emp{not} exist different densities $\vrho,\la$ for which \eqref{eq:couplingstrict} would hold for every $x\geq0$ simultaneously.
\end{proposition}
Nevertheless, our main results (Section \ref{sec:mainr}) and our techniques do \emp{not} require the existence of the measure $\mub^{\vrho,\,\la}$, in particular that of $\nu^{\vrho,\,\la}$, and we do not need to assume \eqref{eq:couplingstrict}. In fact we do not necessarily need to start with stationary marginals.

\subsection{The distribution of the second class particle}\label{sec:scpdistr}
We spell out our fundamental result which will combine with hydrodynamics (to be explicated in Section \ref{sec:hydro}) to give the main Theorems \ref{thm:scphydroasym} and \ref{thm:scpdiffusivelim}. It connects the law of the displacement of a single second class particle with that of a (first class) particle occupation variable. Fix $\vrho>\la$, and recall the initial distributions \eqref{eq:initprodmeas} for a single model and \eqref{eq:scpstart} for a pair with the second class particle.
\begin{theorem}[Displacement distribution of the second class particle]\label{thm:scpdistr}
Suppose that a family of measures $\nu$ fulfills Assumption \ref{assump:marginalmeasure}.
Then for any $n\in\Zbb$ and $t\in\Rbb^+_0$ we have
\begin{equation}\label{eq:findispldistr}
\hat\Prob\{Q(t) \leq n\}=\frac{\vrho - \Exp_{\sigb^{\vrho,\,\la}}\om_{n+1}(t)}{\vrho - \la}.
\end{equation}
\end{theorem}
Note that $\nu$ does not have to be related to the stationary distributions of $\omb$ or of $(\hat{\omb},\hat{\etab})$ in any way. Also observe that Theorem \ref{thm:scpdistr} holds regardless of whether the family of measures $\nu$ satisfy the property detailed in Proposition \ref{prop:probcouplingexist} above.

Furthermore, notice that we had no further assumptions on the rates $p$ and $q$, hence \emp{both} asymmetric and symmetric processes are included in the above assertion. Indeed, a careful overview of our technique (see the proof of Theorem \ref{thm:scpdistr}) reveals that \eqref{eq:findispldistr} also holds for those models with long range jumps or with (non-)finite range dependent rates.

A rather classical result immediately follows from Theorem \ref{thm:scpdistr}, namely the quantity $\Exp_{\sigb^{\vrho,\,\la}}\om_n(t)$ has uniform lower and upper bounds $\la$ and $\vrho$, respectively, in the space $(n,t)\in\Zbb\times\Rbb_0^+$. Also observe that for each fixed $t\in\Rbb_0^+$, the function $n\mapsto\Exp_{\sigb^{\vrho,\,\la}}\om_{n}(t)$ is \emp{monotone non-increasing} in $n\in\Zbb$.

\subsection{The site of the second class particle}\label{sec:bg}
Simple exclusion is special in many ways. One of its simplifying feature is due to $\omax=\omin+1$: there is no choice for the configuration at the site of the second class particle. We deterministically have $\om_{Q(t)}(t)=1$ and $\eta_{Q(t)}(t)=0$ for all $t\geq0$. There are more options when $\omax>\omin+1$, and the next theorem gives an interesting result on the site $Q(t)$ of the second class particle in such models. We take any function $\vphi:\Ic\to\Rbb$ for which either condition $\vphi\geq 0$, or $\sum_{y\in\Ic}\abs{\vphi(y)}<+\infty$ holds. Then we define $\Phi(x) = \sum_{y=\omin}^x \vphi(y)$ and further assume $\Exp_{\sigb^{\vrho,\,\vrho}}\Phi(\om_0(t))<+\infty$.
\begin{theorem}[Background as seen from the position of the second class particle]\label{thm:scpbackground}
Suppose that a family of measures $\nu$ fulfills Assumption \ref{assump:marginalmeasure}. Then we have the identity:
\begin{equation}\label{eq:scpbackgroundformula}
\hat\Exp\,\vphi\big(\hat\omb_{Q(t)}(t)\big) = \frac{\Exp_{\sigb^{\vrho,\,\vrho}}\Phi(\om_{0}(t))-\Exp_{\sigb^{\la,\,\la}}\Phi(\om_{0}(t))}{\vrho-\la}.
\end{equation}
\end{theorem}
In plain words, this theorem tells that the law of $\hat\om_{Q(t)}(t)$ for a $t\geq0$, i.e.\ the particle occupation number at the position of the second class particle, can fully be captured by that of $\om_0(t)$ of $\omb$ starting from $\sigb^{\vrho,\,\vrho}$ and then $\sigb^{\la,\,\la}$.
In particular, if $\sigb^{\vrho,\,\vrho}$ and $\sigb^{\la,\,\la}$ are stationary distributions for the dynamics \eqref{eq:infgen} then the background marginal one-site process $(\hat\om_{Q(t)}(t),\hat\eta_{Q(t)}(t))_{t\geq0}$, as seen from the position of the lone second class particle, is stationary. This can be thought of as another fact proving the intrinsicality of the marginal $\hat\nu^{\vrho,\,\la}$. Notice though that Theorem \ref{thm:scpbackground} does not say anything about the distribution of any site other than that of the second class particle, those are in general \emp{not} stationary. A few very special cases of joint stationary distributions seen by the second class particle are described in \cite{dls,valak,rwshscp} and references therein.

\section{The misanthrope family}\label{sec:misa}

In this section we briefly discuss a special class of attractive particle systems called the misanthrope family where our main results naturally apply. We again underline that there is a much larger class of processes (and initial measures) that we also cover.

First, define the \emp{Gibbs measures} as
\begin{equation}\label{eq:gibbs}
\Gamma^{\theta}(x) :\,= \frac{1}{Z(\theta)}\cdot\exp\big(\,\theta\cdot x + E(x)\, \big) \qquad (x\in\Ic),
\end{equation}
where $\theta\in\Rbb$ is a generic real parameter, which is often referred to as the \emp{chemical potential};  $E:\Ic\to\Rbb$ is any function with appropriate asymptotic growth; finally, the \emp{statistical}- or \emp{partition sum} is $Z(\theta) = \sum_{y\in\Ic}\exp\big(\,\theta \cdot y + E(y)\big)$.

It is known that the above defined Gibbs measures satisfy Assumption \ref{assump:marginalmeasure} (see \cite[Appendix A]{exists} and also \cite{convex}). For the sake of completeness we restate this result below.
\begin{proposition}
Assume that $\Gamma:\,=\big(\Gamma^{\theta}\big)_{\theta\in\Dc_c}$ forms a bunch of probability measures with finite variance, where $\Dc_c$ is some open set of the reals. Then $\Gamma$ satisfies Assumption \ref{assump:marginalmeasure}. In particular, there is a bijection between the parameters $\theta\in\Dc_c$ and the densities $\vrho = \vrho(\theta)\in\Rbb$; and for $\theta(\la)<\theta(\vrho)$, or equivalently for $\la<\vrho$, the measure $\Gamma^{\theta(\vrho)}$ stochastically dominates $\Gamma^{\theta(\la)}$.
\end{proposition}
Due to the bijection claimed in the previous assertion we will change freely between the representations of the measure \eqref{eq:gibbs} either by the chemical potential $\theta=\theta(\vrho)$ or by the density $\vrho=\vrho(\theta)$.

We emphasize that $\Gamma^{\vrho}$ is \emp{not} necessarily a stationary marginal of the dynamics \eqref{eq:infgen} in general. Following ideas of Cocozza-Thivent \cite{coco}, for attractive systems, where $\Gamma^{\vrho}$ is indeed stationary, a nice characterization theorem was established by M.\ Bal\'azs et al., which we recall in the following.
\begin{theorem}[M.\ Bal\'azs et al.]\label{thm:prodstatergod}
Let
\begin{equation*}
E(x) = \sum_{y = x+1}^0\log(f(y))\, -\, \sum_{z=1}^{x}\; \log(f(z))\qquad (x\in\Ic),
\end{equation*}
where $f:\Zbb\to\Rbb^+$ is such that $f(x)=1$ whenever $x\in\Zbb\backslash \Ic$, and is monotone non-decreasing on $\Ic\backslash\{\omin\}$. (The empty sum is as usual defined to be zero.)
Suppose furthermore that:
\begin{itemize}
\item there are \emp{symmetric functions} $s_p,s_q:\Ic\times\Ic\to\Rbb^+_0$ such that
\begin{equation}\label{eq:specrates}
\begin{aligned}
p(\om_i,\om_{i+1}) &= s_p(\om_i,\,\om_{i+1}+1)\cdot f(\om_i)\quad (\omb\in\Omega);\\
q(\om_i,\om_{i+1}) &= s_q(\om_i+1,\,\om_{i+1})\cdot f(\om_{i+1})\quad (\omb\in\Omega),
\end{aligned}
\end{equation}
where $s_p(\omin,\,\cdot\,)\equiv s_p(\,\cdot\,,\omax)\equiv s_q(\,\cdot\,,\omin)\equiv s_q(\omax,\cdot\,)\equiv 0$ holds whenever $\omin$ or $\omax$ is finite, otherwise they are non-zero except when $p$ or $q$ is set to be zero (totally asymmetric case);
\item for any $\omb\in\Omega$ and $i\in\Zbb$:
\begin{multline}\label{eq:combiassumponrates}
p(\om_i,\om_{i+1}) + p(\om_{i+1},\om_{i+2}) + p(\om_{i+2},\om_i)+q(\om_i,\om_{i+1}) + q(\om_{i+1},\om_{i+2}) + q(\om_{i+2},\om_i)\\
=p(\om_i,\om_{i+2}) + p(\om_{i+2},\om_{i+1}) + p(\om_{i+1},\om_i)+q(\om_i,\om_{i+2}) + q(\om_{i+2},\om_{i+1}) + q(\om_{i+1},\om_i).
\end{multline}
\end{itemize}
Then the density parametrized product measure $\Gmb^{\vrho}:\,=\bigotimes_{i=-\infty}^{+\infty}\Gamma^{\vrho}$ is \emp{extremal} among the trans\-la\-tion-invariant stationary distributions of the process with rates $p,\,q$ and infinitesimal generator $\Gc$ of \eqref{eq:infgen}.
\end{theorem}
\begin{remark}
The conditions of Theorem \ref{thm:prodstatergod} originate a wide range of attractive models which we call the \emp{misanthrope family} of processes throughout the article. We will discuss some in Section \ref{sec:partmodels}.
\end{remark}
\begin{remark}
We underline that neither \eqref{eq:specrates} nor \eqref{eq:combiassumponrates} is a requirement for any of our results in Section \ref{sec:furth}.
\end{remark}
\begin{remark}
The stationarity part of Theorem \ref{thm:prodstatergod} has been carried out thoroughly in \cite{varj2nd}, in which all the extremal translation-invariant stationary distributions were covered by examining the convergence region of the partition sum $Z$. For the ergodicity we will briefly comment on how Lemmas 7.2 and 7.3 of \cite{exists} established for the bricklayers' process can be modified to be handy for any process. First, it is not hard to see that Lemma 7.2 can be extended to the cases when (in any order) a positive and a negative second class particle start from next to each other. This results in that the probability of them colliding before any given time is (strictly) positive. Here the only required property of the underlying process is the continuity of its semigroup. Then in Lemma 7.3 ergodicity is carried out by showing that any invariant $\Lc^2$ function $\psi$ w.r.t.\ $\Gmb^{\theta(\vrho)}$ is constant. Now, by using (the extended version of) Lemma 7.2 it can be easily pointed out that adding $(+1,-1)$ (or $(-1,+1)$) to adjacent occupation numbers, whenever this change keeps the state space, does not modify the value of an invariant $\psi$. It follows that interchanging any two adjacent sites does not change the value of $\psi$ under $\Gmb^{\theta(\vrho)}$. The argument is then completed by the application of the Hewitt--Savage $0$-$1$ law.
\end{remark}
Finally, in the above particular case \eqref{eq:gibbs}, consider the measure $\hat\nu^{\vrho,\,\la}$ of \eqref{eq:scpcouplatorig} that is:
\begin{align*}
\hat\nu^{\vrho,\,\la}(x,y)
&=\frac{1}{\vrho-\la}\sum_{z=\omin}^y\big(\Gamma^{\theta(\la)}(z)-\Gamma^{\theta(\vrho)}(z)\big)\cdot\ind\{x=y+1\}\\
&=\frac{\theta(\vrho)-\theta(\la)}{\vrho-\la}
\cdot\sum_{z=y+1}^{\omax}\frac{\Gamma^{\theta(\vrho)}(z)-\Gamma^{\theta(\la)}(z)}{\theta(\vrho)-\theta(\la)}
\cdot\ind\{x=y+1\}
\end{align*}
for $x,y\in\Ic$. Now, fixing $\vrho$ and taking the limit as $\la\uparrow\vrho$ we obtain
\[
(\hat\nu^{\vrho})'(x,y) :\,= \theta'(\vrho)\cdot\sum_{z=y+1}^{\omax} (z-\vrho)\cdot\Gamma^{\theta(\vrho)}(z)\cdot\ind\{x=y+1\}\qquad(x,y\in\Ic),
\]
where it is easy to see that $\theta'(\vrho)=\frac{1}{\Var(\om_0)}$ for $\om_0$ distributed as $\Gamma^{\theta(\vrho)}$. (The empty sum is defined to be zero.) Observe that this probability measure $(\hat\nu^\vrho)'$ is just the marginal at the origin of the initial distribution that was used in \cite[Theorem 2.2]{varj2nd} to start a single second class particle from that position. Thus our treatment is in correspondence with results from \cite{varj2nd}. As a side remark we mention without details that via a second order Taylor expansion as $\la\uparrow\vrho$ one can formally recover the covariance formula in \cite[Theorem 2.2]{varj2nd} directly from \eqref{eq:findispldistr}. Bounding the error terms that arise is straightforward when $\abs{\Ic}<+\infty$, making this argument rigorous.

\section{Hydrodynamics}\label{sec:hydro}

This section is devoted to briefly recall the main notions and results from hydrodynamics of asymmetric and symmetric particle systems as well. Some of the results below use the misanthrope class (see the previous Section \ref{sec:misa}) while others are more general.

\subsection{Hydrodynamics of asymmetric models}\label{sec:asymhydro}
The idea behind the hydrodynamic limit for asymmetric systems is that, in \emph{hyperbolic} scaling (i.e.\ same scale for space and time), the rescaled microscopic average density of interacting particles behaves as a deterministic density field obeying the conservation law
\begin{equation}\label{eq:hydrogeneral}
\left.
\begin{aligned}
\partial_t u + \partial_x G(u)&= 0\\
u(\,\cdot\,,0) &= v(\cdot)
\end{aligned}
\right\}
\end{equation}
where $u=u(x,t)$ is the \emp{(macroscopic) density} with initial condition $v(\cdot)$. The function $G$ is called the \emp{hydrodynamic flux} and is
\begin{equation}\label{eq:hydrofluxdef}
G(\vrho)=\Exp_{\pib^{\vrho}}\big[p(\om_0,\,\om_{1})-q(\om_0,\,\om_{1})\big].
\end{equation}
Here $\Exp_{\pib^{\vrho}}$ denotes the expectation w.r.t.\ the extremal stationary distribution $\pib^{\vrho}$ for a density $\vrho$.

The \emp{rescaled empirical measure} of a sequence of random configurations $(\omb^N)_{N\in\Nbb}$ is defined as
\begin{equation*}
\alpha^N\big(\omb^N,\di x\big)=\frac{1}{N}\sum_{j\in\Zbb}\om_j^N\ind\big\{j/N\in\di x\big\}\qquad(N\in\Nbb).
\end{equation*}
A deterministic bounded Borel measurable function $v$ on $\Rbb$ is the \emp{density profile} of $(\omb^N)_{N\in\Nbb}$, if $\alpha^N(\omb^N,\di x)$ converges to $v(x)\di x$ as $N\to+\infty$ for all $x\in\Rbb$, in probability as a random object, and in the topology of vague convergence as a measure, meaning that
\begin{equation*}
\lim_{N\to+\infty}\Prob^N
\bigg(\bigg|\frac{1}{N}\sum_{j\in\Zbb}\psi\big(j/N\big)\cdot\om_j^N-
\int_{x\in\Rbb}\psi(x)\cdot v(x) \di x\bigg| > \veps\bigg)=0
\end{equation*}
is required to hold for each $\veps>0$ and with all continuous test function $\psi:\Rbb\to\Rbb$ of compact support.
\begin{definition}[Hydrodynamic limit]\label{def:hydrolimit}
A sequence of processes $(\omb^N(t))_{t\geq 0}$ ($N\in\Nbb$), all generated by $\Gc$ of \eqref{eq:infgen} with random initial configurations $\omb_0^N$ ($N\in\Nbb$) \emp{exhibits a hydrodynamic limit} $u$, if $(\omb^N(N t))_{N\in\Nbb}$ has density profile $u(\,\cdot\,,t)$ for every $t\geq 0$, where $u$ is a (weak) solution to the problem \eqref{eq:hydrogeneral}.
\end{definition}
We note that the hydrodynamic limit just defined is also referred to as the \emp{weak conservation of local equilibrium} (cf.\ \cite[Chapter 4]{cl}). Finally, we introduce one more notation: for a fixed $n\in\Zbb$ denote by $\tau_n$ the \emp{shift operator} which acts on a configuration $\omb\in\Omega$ as $(\tau_n\omb)(i) = \om_{i+n}$ $(i\in\Zbb)$ and on a measure $\kpb:\Omega\to[0,1]$ as $\tau_n\kpb(\omb) = \kpb(\tau_n\omb)$, respectively.

In the following, we will make the choice $\sigb^{\vrho,\,\la}$ of \eqref{eq:initprodmeas} as a common initial distribution for the sequence of processes to be rescaled in the hydrodynamic limit. It follows that the limiting process has the \emp{Riemannian (step) initial density profile}
\begin{equation}\label{eq:stepinitcond}
v(x)=\left\{
\begin{aligned}
&\vrho, && \text{if } x\leq 0, \\
&\la,   && \text{if } x>0.
\end{aligned}
\right.
\end{equation}
Under Assumption \ref{assump:marginalmeasure} and mild assumptions on the flux function $G$ there exists a \emp{unique entropy solution} $u$ to the problem \eqref{eq:hydrogeneral} with \eqref{eq:stepinitcond} as initial condition. It is also known that for each $t\geq0$, this weak solution is continuous apart from a finite set of jump discontinuities (shocks), where we define $u(\,\cdot\,,t)$ to be left-continuous. For concepts and results in hyperbolic conservation laws, which were omitted here, we refer to \cite{bagurasa} and further references therein (see also \cite{holdenrisebro}).

In what follows some exact results on hydrodynamics will be collected concerning the above setting. The first general result is from \cite{hl}, valid in the misanthrope framework of Section \ref{sec:misa} (see Theorem \ref{thm:prodstatergod}).
\begin{theorem}[F.\ Rezakhanlou]\label{thm:rezakhydro}
Take any process from the misanthrope family equipped with \emp{bounded} rates. Set the initial measure $\sigb^{\vrho,\,\la}$ to be of \emp{stationary} marginals. Then $(\omb^N(t))_{t\geq 0, N\in\Nbb}$ exhibits a hydrodynamic limit $u$, where $u$ is the unique entropy solution to \eqref{eq:hydrogeneral} with hydrodynamic flux $G$ \eqref{eq:hydrofluxdef} and with initial datum \eqref{eq:stepinitcond}. In addition, the limit
\begin{equation}\label{eq:rezakexpvlimit}
\lim_{N\to+\infty}\Exp_{\sigb^{\vrho,\,\la}}\bigg[\frac{1}{N}\sum_{j\in\Zbb}\psi(j/N)\cdot \vphi(\tau_j\omb(N t))\bigg] =
\int_{x\in\Rbb}\psi(x)\cdot \Exp_{\Gmb^{u(x,t)}}\vphi(\omb) \di x
\end{equation}
also holds for every continuous $\psi$ of compact support and any cylinder function $\vphi:\Omega\to\Rbb$.
\end{theorem}
In the above result we are much restricted for the marginals of the initial measure to be chosen properly. However, this was far more generalized by C.\ Bahadoran et.\ al.\ in \cite{bagurasa} for systems of bounded particle numbers per site. In particular, this result does not require the special algebraic structure of Section \ref{sec:misa}.
\begin{theorem}[C.\ Bahadoran, H.\ Guiol, K.\ Ravishankar and E.\ Saada]\label{thm:bagurasa}
Suppose that Assumption \ref{assump:marginalmeasure} holds and that both $\omin$ and $\omax$ are finite. Then $(\omb^N(t))_{t\geq 0}$ exhibits a hydrodynamic limit $u$ for every $\sigb^{\vrho,\,\la}$ of \eqref{eq:initprodmeas} with some Lipschitz continuous hydrodynamic flux $G$, where $u$ is the unique entropy solution to \eqref{eq:hydrogeneral} with initial datum \eqref{eq:stepinitcond}.
\end{theorem}
\begin{remark}
We refer to \cite{bagurasa} for the detailed definition of $G$ in the general case. Indeed, the previous assertion holds in even more general context as well as with sharper conclusions, for details consult \cite{bagurasa} and \cite{bagurasastrong}.
\end{remark}
\begin{remark}\label{rem:bagurasaunbounded}
Thanks to the step initial condition, by \cite[Remark 2., pp. 1347]{bagurasa}, we can extend Theorem \ref{thm:scphydroasym} for those \emp{unbounded systems} described in Theorem \ref{thm:prodstatergod} where the rates are bounded. We understand from informal communications that these results can further be generalized to models with unbounded rates as well.
\end{remark}

Our ultimate goal would be to conclude that the rescaled quantity $\Exp_{\sigb^{\vrho,\,\la}}\om_{[N x]}(N t)$ also converges, where $[\,\cdot\,]$ denotes the integer part function. This, however, does not appear to be an immediate consequence of the above theorems. But C.\ Landim \cite{landim} has elaborated a set of assumptions under which this consequence eventually holds (note also \cite[Proposition 0.6, Chapter 6]{cl}). We are going to recapitulate this result below to be formulated in our special context with sharper conclusions, outlining its proof in Section \ref{sec:proof}.
\begin{proposition}\label{prop:landimlocconv}
Suppose that the process with infinitesimal generator \eqref{eq:infgen} exhibits a density parametrized, stochastically ordered and \emp{continuous} family $(\pib^\vrho)_{\vrho\in \Rc}$ of trans\-la\-tion-invariant stationary distributions, where $\Rc\subset\Rbb$ is such that $\vrho_{\min}<\vrho_{\max}\in\Rc$. Fix a cylinder function $\vphi:\Omega\to\Rbb$, being either bounded or monotone non-decreasing, such that $\Exp_{\pib^{\vrho_{\max}}}\abs{\vphi}(\omb)<+\infty$. Assume furthermore that the convergence
\begin{equation}\label{eq:landimexpvlimit}
\lim_{N\to+\infty}\Exp_{\tau_{[N\veps]}\sigb^{\vrho,\,\la}}\bigg[\frac{1}{N}\sum_{j\in\Zbb}\psi(j/N)\cdot \vphi\big(\tau_j\omb^{\veps,N}(N t)\big)\bigg]
=\int_{x\in\Rbb}\psi(x)\cdot \Exp_{\pib^{u^{\veps}(x,t)}}\vphi(\omb) \di x
\end{equation}
takes place for every $\veps\in\Rbb$ and continuous $\psi:\Rbb\to\Rbb$ of compact support with some uniformly bounded  family of functions $(u^{\veps})_{\veps\in\Rbb}$ for which
$u^{\veps}:\Rbb\times\Rbb^{+}_0\to[\vrho_{\min},\vrho_{\max}]$ is monotone non-increasing for each fixed $t\in\Rbb^{+}_0$ and for every continuity point $x\in\Rbb$ of $u^0(\,\cdot\,,t)$: $\lim_{\veps\to 0}u^{\veps}(x,t)=u^{0}(x,t)$. Then we have
\begin{equation}\label{eq:stronghdl}
\lim_{N\to+\infty}\Exp_{\sigb^{\vrho,\,\la}}\vphi\big(\tau_{[N x]}\omb(N t)\big) = \Exp_{\pib^{u^0(x,t)}}\vphi(\omb)
\end{equation}
for every continuity point $x$ of $u^0(\,\cdot\,,t)$.
\end{proposition}
By continuity of the set $(\pib^{\vrho})_{\vrho\in\Rc}$ we mean that if $\vrho_n\to\vrho$ as $n\to+\infty$, where $\vrho_n,\vrho\in\Rc$, then $\pib^{\vrho_n}\to\pib^{\vrho}$ in the weak sense. Furthermore, the monotonicity of $\vphi$ preserves the coordinate-wise order of configurations $\omb,\etab\in\Omega$, that is if $\omb\geq\etab$ then $\vphi(\omb)\geq\vphi(\etab)$. The convergence in \eqref{eq:stronghdl} is also called the \emp{conservation of local equilibrium} (cf.\ \cite[Chapter 1]{cl}).

\subsection{Hydrodynamics of symmetric models}\label{sec:symhydro}
In our context, being symmetric means $q(\om_i,\om_{i+1})=p(\om_{i+1},\om_i)$ for each $\omb\in\Omega$. Note that attractiveness is still up, that is $p$ is required to be monotone non-increasing (non-decreasing) in its first (second) variable. We say that a symmetric attractive process is \emp{gradient} if there exists a cylinder function $g:\Omega\to\Rbb$ for which
\begin{equation}\label{eq:gradientcondition}
p(\om_i,\om_{i+1}) - p(\om_{i+1},\om_i) = g(\tau_{i}\omb) - g(\tau_{i+1}\omb)
\end{equation}
holds for every $i\in\Zbb$ and $\omb\in\Omega$ (recall the shift operator $\tau$ after Definition \ref{def:hydrolimit}). Usually it is more convenient and simpler to deal with \emp{attractive gradient systems}. For such systems the key quantity turns out to be the \emp{diffusivity coefficient} $d$ which is defined to be $d(\vrho)=\Exp_{\pib^{\vrho}}g(\omb)$, where $\pib^{\vrho}$ is a stationary distribution of the process with density $\vrho$. Note the difference between $d$ and the hydrodynamic flux $G$ \eqref{eq:hydrofluxdef} being its hyperbolic counterpart.

The concepts of hydrodynamics of the previous subsection can be exactly repeated here, except that this time the relevant scaling is \emp{diffusive} instead of hyperbolic. Hence the macroscopic behavior of the density field is described by a \emp{parabolic} partial differential equation of the form
\begin{equation}\label{eq:symhydrogeneral}
\left.
\begin{aligned}
\partial_t u &= \text{{\small$\frac{1}{2}$}}\, \Delta d(u)\\
u(\,\cdot\,,0) &= v(\cdot)
\end{aligned}
\right\}
\end{equation}
where $v$ is the step function defined in \eqref{eq:stepinitcond}. In general it is not so obvious for \eqref{eq:symhydrogeneral} to have a unique bounded (classical or weak) solution due to the discontinuity of $v$ and the smoothness of $d$. We skip investigating this issue by assuming that there always exists a unique weak solution to \eqref{eq:symhydrogeneral}. The hydrodynamic limit of gradient systems is well known (see \cite{cl}, \cite[Chapter 8]{transinvexcl} and many references therein, particularly \cite{ferrpresuttivaressymzr} and \cite{landiminfinite}) and methods partially extend to non-gradient systems \cite{perrut} as well.

\section{Particular examples}\label{sec:partmodels}

We have selected some particular models in order to demonstrate the versatility of our results. Our general framework contains several well studied examples like the (totally) asymmetric simple exclusion process or the class of zero-range processes. We first list asymmetric and then symmetric processes with additional descriptions. Once for all we fix two reals $\bar{p},\bar{q}$ for which $0\leq \bar{p}<\bar{q}\leq 1$ and $\bar{p}+\bar{q}=1$ hold.
\paragraph{Generalized exclusion processes}
Many systems with bounded occupations lie in this class but we only illustrate two of them.

The first one is the \emp{$2$-type model} of \cite{b_n_t_t_coex16}, which is a totally asymmetric process with $\Ic=\{-1,0,+1\}$ and rates
\begin{equation}\label{eq:2typemodelrates}
p(0,0) = c,\qquad p(0,-1)=p(+1,0) = \frac{1}{2}, \qquad p(+1,-1) = 1
\end{equation}
and $q\equiv 0$. The dynamics consists of the following simple rules: two adjacent holes can produce an antiparticle-particle pair (\emp{creation}), (anti)particles can hop to the (negative) positive direction (\emp{exclusion}), and when a particle meets an antiparticle they can annihilate each other (\emp{annihilation}). The process is attractive if and only if $c\leq \frac{1}{2}$ and lies in the range of Theorem \ref{thm:prodstatergod}.

The hydrodynamic behavior, but \emp{not} the second class particles, of the model has been thoroughly investigated by M.\ Bal\'azs, A.\ L.\ Nagy, B.\ T\'oth and I.\ T\'oth \cite{b_n_t_t_coex16}. In that article the hydrodynamic flux $G$ was explicitly calculated, which turned out to be neither concave nor convex in some region of the parameter space. Hence the entropy solution of the hydrodynamic equation can produce various mixtures of rarefaction fans and shock waves. By \eqref{eq:scpspeedlim} it implies that the limit distribution of the second class particle can have both continuous and discrete parts which will be demonstrated in the following. Using the results of \cite{b_n_t_t_coex16} we can basically evaluate \eqref{eq:scpspeedlim} of Theorem \ref{thm:scphydroasym} in each case but we highlight only two of them below. For sake of simplicity we let $\vrho=1$ and $\la=-1$.
\paragraph{Concave flux} In the region $\frac{1}{16}\leq c\leq \frac{1}{2}$: the hydrodynamic flux $G$ is concave \cite{b_n_t_t_coex16}. In particular for $c=\frac{1}{16}$, Figure \ref{fig:riemannconcave} demonstrates how the one parameter family of limit distributions of the second class particle evolves in time as $t\in[0,1]$. We notice that for all $t>0$ the cumulative distribution function $F_t^{Q}$ is \emp{continuous} but has a vertical ``slope'' at the origin. Thus its density is unbounded around zero.
\begin{figure}[!htb]
\centering
\includegraphics[width=0.8\textwidth]{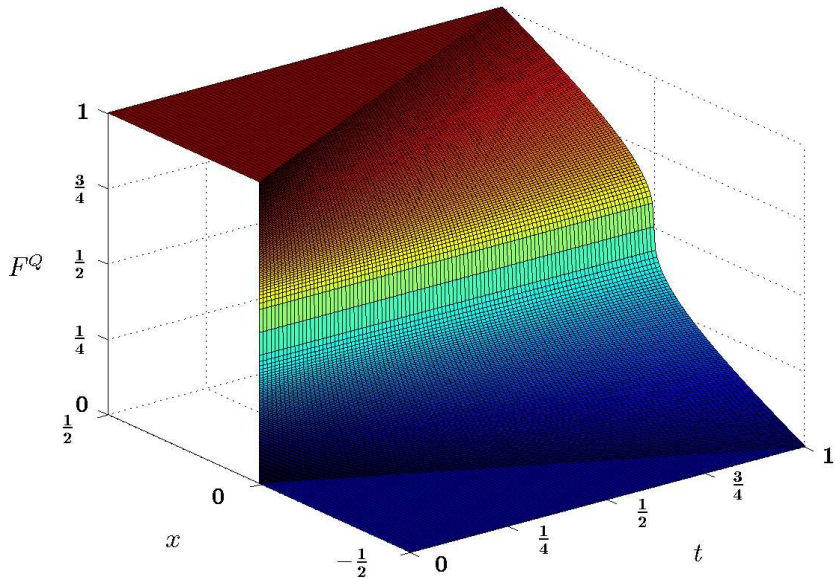}
\caption[Limit distribution of SCP in 2-type model -- concave case]{The limit distribution of second class particle when $c=\frac{1}{16}$, the hydrodynamic flux $G$ being (non-strictly) concave. In particular, a vertical slice of the surface gives a limit distribution $F^Q_t(\,\cdot\,):\,=\lim_{N\to+\infty}\hat\Prob\{\frac{1}{N}Q(Nt)\leq \cdot\}$ for a fixed $t$.}\label{fig:riemannconcave}
\end{figure}

\paragraph{Non-convex flux} In the region $0<c<\frac{1}{16}$: the hydrodynamic flux $G$ is neither concave nor convex \cite{b_n_t_t_coex16}. As a particular example, for $c = \frac{1}{324}$ the model can develop a (non-linear) \emp{rarefaction fan -- shock -- rarefaction fan profile} in the hydrodynamic limit. The second class particle then may stick into the shock with probability $v_{\max}:\,=\frac{1}{2}\sqrt{\frac{1-16c}{1-4c}}$ or it follows a continuously chosen characteristics in one of the regions of the rarefaction fan. Figure \ref{fig:riemannnonconvex} demonstrates this behavior as $t\in[0,1]$.
\begin{figure}[!htb]
\centering
\includegraphics[width=0.8\textwidth]{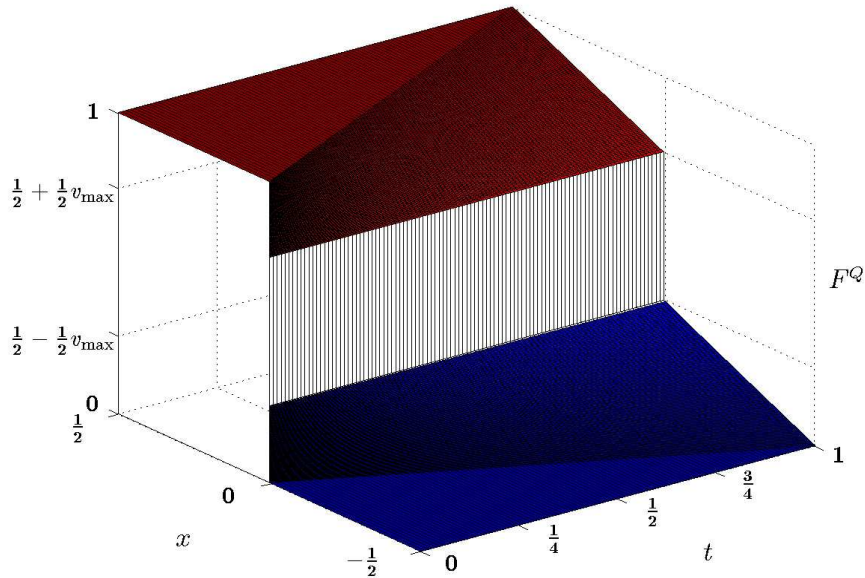}
\caption[Limit distribution of SCP in 2-type model -- non-concave, non-convex case]{The limit distribution of the second class particle in the case of $c = \frac{1}{324}$ when the hydrodynamic flux $G$ is neither concave nor convex. A vertical slice of the surface gives a particular limit distribution.}\label{fig:riemannnonconvex}
\end{figure}

Relying again on \cite{b_n_t_t_coex16} we finally note that one can explicitly calculate the estimate of Theorem \ref{thm:collprob}: the collision of two second class particles starting from the rarefaction fan has at least probability $C_0 = G(0) = \frac{\sqrt{c}}{1+2\sqrt{c}}\in(0,\frac{1}{2+\sqrt{2}}]$ in this model.

Another example we highlight is the \emp{$K$-exclusion process}, where $K$ is any positive integer. Set $\Ic=\{0,1,\ldots,K-1,K\}$ and let the rates be
\begin{equation*}
p(\om_i,\om_{i+1}) = \bar{p}\cdot\ind\{\om_i>0; \om_{i+1}<K\},\;
q(\om_i,\om_{i+1}) = \bar{q}\cdot\ind\{\om_{i+1}>0; \om_i<K\}.
\end{equation*}
In particular for $K=1$ we obtain the asymmetric simple exclusion process with the family of Bernoulli product measures as extremal translation-invariant stationary distributions which work well for Assumption \ref{assump:marginalmeasure} and thus recover the result of \cite{serf}.
For $K>1$ much less is known (the assumptions of Theorem \ref{thm:prodstatergod} cease to hold). In particular, it is not known whether its density parametrized translation-invariant extremal stationary distributions span the range $[0,K]$. They are proved to exist for some closed parameter set $\Rc\subset[0,K]$ (see \cite[Corollary 2.1]{bagurasa}). The structure of these measures is also unknown. The model, however, exhibits a hydrodynamic limit resulting in a conservation law with a \emp{concave} flux $G$ (see Theorem \ref{thm:bagurasa} and also \cite{hkl}). For $G$ only some qualitative properties have been established (see \cite{hkl}).
Nevertheless, one can still apply Theorem \ref{thm:scphydroasym} with any product initial distribution that satisfies Assumption \ref{assump:marginalmeasure}.
\paragraph{Zero range processes} Let $\omin=0$ and $\omax=+\infty$. The jump rates are defined as $p(\om_i,\om_{i+1}) = \bar{p}\cdot f(\om_i)$ and $q(\om_i,\om_{i+1}) = \bar{q}\cdot f(\om_{i+1})$, where $f:\Zbb^+\to\Rbb^+$ is a monotone non-decreasing function with at most linear growth and with $f(0)=0$. This family satisfies all the assumptions of Theorem \ref{thm:prodstatergod}, hence hydrodynamics follows via Theorem \ref{thm:rezakhydro}. For sake of brevity we only spell out two totally asymmetric examples ($\bar{p} = 1$). The hydrodynamic flux is then given by $G(\vrho)=\exp(\theta(\vrho))$. The two most well-known special cases we consider are the ones of constant and linear rates: $f(x)=\ind\{x>0\}$ and $f(x)=x$, respectively. In the former case the extremal translation-invariant stationary distributions are of product form with geometric site-marginals while in the latter the Poisson distribution takes over this role. $G(\vrho)$ is $\vrho\cdot(1+\vrho)^{-1}$ and $\vrho$, respectively.

A straightforward computation then shows (see \cite[Section 2.2, pp. 30--36]{holdenrisebro}), that for the totally asymmetric constant rate zero-range process, \eqref{eq:scpspeedlim} takes the form:
\[
\lim_{N\to+\infty}\hat\Prob\bigg\{\frac{Q(Nt)}{N}\leq x\bigg\} =
\left\{
\begin{aligned}
&0 &&
 \text{if }\textstyle{(1+\vrho)^2\leq \frac{t}{x}};\\
&\textstyle{\frac{\vrho+1}{\vrho-\la}-\frac{1}{\vrho-\la}\sqrt{\frac{t}{x}}}&&
\text{if }\textstyle{(1+\la)^2\leq \frac{t}{x}<(1+\vrho)^2};\\
&1 &&
\text{if }\textstyle{(1+\la)^2 > \frac{t}{x}}.
\end{aligned}
\right.
\]
Due to Remark \ref{rem:bagurasaunbounded} we in fact get this law for the limit velocity of the second class particle no matter how we choose, still under Assumption \ref{assump:marginalmeasure}, the initial marginals of \eqref{eq:initprodmeas}.

The totally asymmetric linear rate zero-range process is a much easier story (system of independent walkers). In that case $Q$ is a unit rate Poisson process in agreement with the unique entropy solution of the transport equation $\partial_t u +\partial_x u = 0$, being $u(x,t)=u_0(x-t)$. No novelty here, of course.
\paragraph{Deposition models} Now, let $\omin=-\infty$ and $\omax=+\infty$. The generalized bricklayers' process is defined to have rates:
\[
p(\om_i,\om_{i+1}) = \bar{p}\cdot(f(\om_i) + f(-\om_{i+1}));\;
q(\om_i,\om_{i+1}) = \bar{q}\cdot(f(-\om_i) + f(\om_{i+1})),
\]
where $f:\Zbb\to\Rbb^+$ is any monotone non-decreasing function, also having the property that $f(x)\cdot f(1-x) = 1$ for all $x\in\Zbb$. This family was first introduced and investigated in \cite{valak}. For rates growing at most exponentially, existence of dynamics was showed in the totally asymmetric case \cite{exists}. However, results concerning hydrodynamics have not been established yet for unbounded rates, and so Theorem \ref{thm:scphydroasym} is conditional in this case.

In particular, we obtain the \emp{totally asymmetric exponential bricklayers' process} if we set $\bar{p} = 1$ and $f(x)=\exp(\beta(x-1/2))$ ($x\in\Zbb$), where $\beta$ is a fixed positive real. Then Theorem \ref{thm:prodstatergod} shows that the product of \emp{discrete Gaussian distributions}, defined as
\begin{equation*}
\Gamma^{\theta(\vrho)}(x)=
\frac{1}{Z(\theta(\vrho))}\cdot\exp\big(-\beta\cdot\big(x-\theta(\vrho)/\beta\big)^2/2\big)\qquad (x\in\Zbb),
\end{equation*}
where $\vrho\in\Rbb$, $\theta\in\Rbb$ and $Z(\theta(\vrho))=\sum_{y\in\Zbb}\exp\big(-\beta\cdot(y-\theta(\vrho)/\beta)^2/2\big)$, is stationary. This measure enjoys the remarkable property that
\begin{equation*}
\Gamma^{\theta(\vrho)}(x)=\Gamma^{\theta(\vrho)-\beta}(x-1)=\Gamma^{\theta(\vrho-1)}(x-1)\qquad (\vrho\in\Rbb, x\in\Zbb).
\end{equation*}
This fact indeed implies that the discrete Gaussian satisfies even condition \eqref{eq:couplingstrict} of Proposition \ref{prop:probcouplingexist}. Articles \cite{valak,sokvalak,rwshscp} made also good use of the previous identity for exploring special random walking shock-like product distributions that also include the second class particle. Finally, if we use $\Gamma^{\theta(\vrho)}$ and $\Gamma^{\theta(\la)}$ as initial marginals with $\vrho=\la+1\in\Rbb$, then the measure $\hat{\nu}^{\vrho,\,\la}$ gets a particularly simple form, namely $\hat\nu^{\vrho,\,\la}(x,y)=\Gamma^{\theta(\la)}(y)\cdot\ind\{x=y+1\}$.
\paragraph{Symmetric exclusion processes} One of the most studied gradient processes is the \emp{simple symmetric exclusion process} with $\Ic=\{0,1\}$ and with rates $p(\om_i, \om_{i+1}) = q(\om_{i+1},\om_{i}) = \om_i\cdot(1-\om_{i+1})$. In this case it is straightforward that the motion of the second class particle is a \emp{simple symmetric random walk}. Hence its scaling results in the normal distribution. Our machinery implies this very simple fact and is clearly an overshoot for this case.

Nevertheless, the scaling limit of the second class particle in our scenario becomes far less trivial for more sophisticated symmetric models like the next one. Let $\Ic=\{-1,0,+1\}$ and define the non-vanishing rates as
\begin{gather*}
p(0, 0) = q(0, 0) = c,\qquad p(+1, -1) = q(-1, +1) = 2,\\
p(0, -1) = q(-1, 0) = p(+1, 0) = q(0, +1) = 1
\end{gather*}
with $0<c\le1$. Note the essential difference between this and the simple symmetric exclusion process from before. We remark that this model is the symmetrized version of the $2$-type model of \cite{b_n_t_t_coex16} we discussed earlier (see \eqref{eq:2typemodelrates}).

The above defined symmetric processes both enjoy product-form stationary distributions by Theorem \ref{thm:prodstatergod}. Their common gradient function is $g(\omb)=\om_0$ that is the macroscopic behavior is described by the homogeneous heat equation $\partial_t u=\frac{1}{2}\,\Delta u$ in both cases. Hence
\[
\lim_{N\to+\infty}\hat\Prob\bigg\{\frac{Q(Nt)}{\sqrt{N}}\leq x\bigg\} = \Phi\bigg(\frac{x}{\sqrt{t}}\bigg),
\]
where $\Phi$ denotes the cumulative distribution function of a standard normal variable.
\paragraph{Symmetric zero-range processes} Let $\omin=0$ and $\omax=+\infty$. Then define the rates as $p(\om_i,\om_{i+1}) = f(\om_i)$ and $q(\om_i,\om_{i+1}) = f(\om_{i+1})$ with a common, monotone non-decreasing function $f:\Zbb^+\to\Rbb^+$ such that $f(0)=0$. In this case the gradient condition (see \eqref{eq:gradientcondition}) works with $g=f$. And so the macroscopic equation reads as $\partial_t u = \frac{1}{2}\Delta d(u)$, where $d(\vrho)=\Exp_{\Gmb^{\vrho}}f(\omb)$ and $\Gmb^{\vrho}$ can be read off from Theorem \ref{thm:prodstatergod}. We highlight the constant and the linear rate case, that is when $f(x)=\ind\{x>0\}$ and $f(x)=x$, respectively. In the latter case we get normal behavior for $Q$, while in the former case
\begin{equation}\label{eq:symzrpde}
\partial_t u =\frac{1}{2}\,\partial_x\bigg(\frac{1}{(1+u)^{2}}\,\partial_x u\bigg)
\end{equation}
has to be solved.
\begin{figure}[!htb]
\centering
\includegraphics[width=0.8\textwidth]{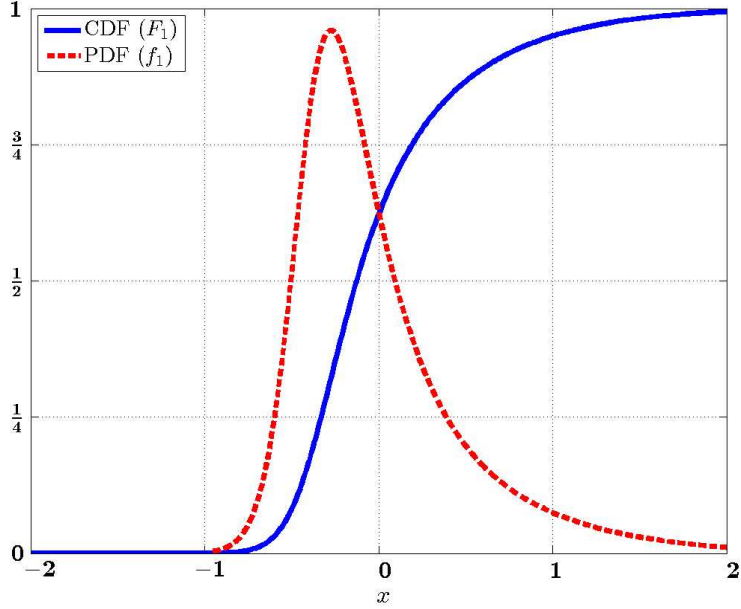}
\caption[Limit distribution of SCP in constant rate symmetric ZR process]{The limit distribution of second class particle at $t=1$ in constant rate symmetric zero-range process with $\la=0$ and $\vrho=3$.}\label{fig:symzrslices}
\end{figure}
By change of variables one can deduce and then easily verify that the unique bounded classical solution of \eqref{eq:symzrpde} with initial condition \eqref{eq:stepinitcond} is $u(x,t)=\big(h^{-1}\big)'\big(x/\sqrt{t}\big)-1$, where
\[
h(y)=\frac{y}{1+\vrho} + \bigg(\frac{1}{1+\la}-\frac{1}{1+\vrho}\bigg)\cdot\big(\vphi(y)+y\cdot\Phi(y)\big)\qquad (y\in\Rbb),
\]
$\vphi$ and $\Phi$ denote the probability density function and the cumulative distribution function of a standard normal variable, while $h^{-1}$ is the inverse function of $h$. This gives that the limit distribution function of \eqref{eq:scpdiffusivelim} is of the form:
\[
F_t(x):\,=\lim_{N\to\infty}\hat\Prob\bigg\{\frac{Q(Nt)}{\sqrt{N}}\leq x\bigg\}=\frac{\vrho+1}{\vrho-\la}-\frac{1}{\vrho-\la}\cdot\frac{1}{h'\big(h^{-1}\big(x/\sqrt{t}\big)\big)},
\]
thus its density function is
\[
f_t(x):\,=\frac{\mathrm{d}}{\mathrm{d}x}\,F_t(x)=\frac{1}{\sqrt{t}}\cdot\frac{1}{(1+\vrho)\cdot(1+\la)}
\cdot\frac{\vphi\big(h^{-1}\big(x/\sqrt{t}\big)\big)}{\big(h'\big(h^{-1}\big(x/\sqrt{t}\big)\big)\big)^3}.
\]
Observe that the limit distribution has always zero expectation though the position of the second class particle is \emp{not} a martingale. Furthermore, $F_{t}$ is positively skewed and that it has negative mode as well as median implying that the second class particle is more likely to have a \emp{negative speed}. In other words, it locally sees more first class particles before him rather than behind him which seems an unexpected phenomenon. An illustration can be found in Figure \ref{fig:symzrslices}.

\section{Proofs}\label{sec:proof}

We first prove the fundamental results of Section \ref{sec:furth}.
\begin{proof}[Proof of Proposition \ref{prop:probcouplingexist}]
Assume that $\vrho-1\leq\la<\vrho$ are given, where $\vrho,\la\in\Dc$.
To be able to construct a joint probability measure $\nu^{\vrho,\,\la}$ with the corresponding marginal distributions $\nu^{\vrho}$ and $\nu^{\la}$ in such a way that $\nu^{\vrho,\,\la}(\{(x,y):x-y\in\{0,1\}\})=1$ also holds we must certainly have that
$\nu^{\vrho,\,\la}(x,y)=0$ whenever $x-y\not\in\{0,1\}$. For sake of simplicity the only non-vanishing probabilities are denoted by
\[
c_{y,y}:=\nu^{\vrho,\,\la}(\{(y,y)\})\quad\text{ and }\quad c_{y+1,y}:=\nu^{\vrho,\,\la}(\{(y+1,y)\}),
\]
where $y\in\Ic$\, ($c_{\omax+1,\omax}$ is defined to be zero).

For the marginals to match we have the following constraints:
\[
c_{y,y}+c_{y,y-1}=\nu^{\vrho}(\{y\})\quad\text{ and }\quad c_{y,y}+c_{y+1,y} = \nu^{\la}(\{y\}),
\]
which implies that $c_{y+1,y}=c_{y,y-1} + \nu^{\la}(\{y\})-\nu^{\vrho}(\{y\})$. It then follows by recursion that
\begin{equation}\label{eq:jointmeasure1}
c_{y+1,y}=\nu^{\la}(\{z:z\leq y\})-\nu^{\vrho}(\{z:z\leq y\}),
\end{equation}
hence we obtain that
\begin{equation}\label{eq:jointmeasure2}
c_{y,y}=\nu^{\vrho}(\{z:z\leq y\})-\nu^{\la}(\{z:z\leq y-1\})
\end{equation}
holds for every $y\in\Ic$.
Since the last two expressions must be non-negative, these provide the stochastic dominance in Assumption \ref{assump:marginalmeasure} and the condition \eqref{eq:couplingstrict}, respectively.

Finally, the identity $\sum_{y\in\Ic}c_{y+1,y}=\vrho-\la$ proves \eqref{eq:condmeasure}.
\end{proof}
\begin{proof}[Proof of Proposition \ref{prop:negativeresult}]
First, in the geometric case, let $0<p<p'\leq1$,
\[
\nu^{\vrho}(\{n\}):=p\cdot(1-p)^n\quad\text{ and }\quad \nu^{\la}(\{n\}):=p'\cdot(1-p')^n \qquad (0\leq n\in\Zbb)
\]
be the weights of the Geometric distributions with mean $\vrho:=1/p-1$ and $\la:=1/p'-1$, respectively.
It is an obvious fact that $\nu^{\vrho}$ stochastically dominates $\nu^{\la}$, that is Assumption \ref{assump:marginalmeasure} holds. But condition \eqref{eq:couplingstrict} is violated, since the inequality
\[
\begin{aligned}
0&
\leq\;\nu^{\vrho}(\{n:n\leq y\})-\nu^{\la}(\{n:n\leq y-1\})\\
&=\sum_{n=0}^y p\cdot(1-p)^n-\sum_{n=0}^{y-1}p'\cdot(1-p')^n\\
&=(1-p')^y\left[1-(1-p')\left(\frac{1-p}{1-p'}\right)^{y+1}\right]
\end{aligned}
\]
cannot hold for all $0\leq y\in\Zbb$ simultaneously.

In the Poisson case, let $0\leq\la<\vrho<+\infty$,
\[
\nu^{\vrho}(\{n\}):=\exp(-\vrho)\cdot\frac{\vrho^n}{n!}\quad\text{ and }\quad \nu^{\la}(\{n\}):=\exp(-\la)\cdot\frac{\la^n}{n!} \qquad (n\in\Zbb^+)
\]
be the weights of the corresponding Poisson distributions. The cumulative distribution functions are of the form
\[
 \nu^{\vrho}(\{n:n\leq y\})=\frac{1}{y!}\int_{\vrho}^{+\infty}s^y\exp(-s)\,\mathrm{d}s\qquad(y\in\Zbb^+)
\]
via integration by parts. This shows the validity of Assumption \ref{assump:marginalmeasure} as well. Condition \eqref{eq:couplingstrict}, however, does not hold, since the following inequality cannot hold for all $0\leq y\in\Zbb$:
\[
\begin{aligned}
0&
\leq\;\nu^{\vrho}(\{n:n\leq y\})-\nu^{\la}(\{n:n\leq y-1\})\\
&=\frac{1}{y!}\left[\int_{\vrho}^{+\infty}s^y\exp(-s)\,\mathrm{d}s-
\int_{\la}^{+\infty}y\cdot s^{y-1}\exp(-s)\,\mathrm{d}s\right]\\
&=\exp(-\la)\cdot\frac{\la^y}{y!}
\left[1-\int_{\la}^{\vrho}\left(\frac{s}{\la}\right)^y\exp(\la-s)\,\mathrm{d}s\right],
\end{aligned}
\]
where at the last inequality we took advantage of the integration by parts formula.
\end{proof}

Next, we define the notion of \emp{particle current} that we will need subsequently. Denote by $\Sbb$ the set of half integers, that is $\Sbb:\,=\Zbb+\oh=\big\{i+\oh:i\in\Zbb\big\}$. For each $\omb\in\Omega$ we assign a \emp{height configuration}
\[
\hb = \big(\ldots,\,h_{-\thh},\,h_{-\oh},\,h_{\oh},\,h_\thh,\,\ldots\big)\in\Zbb^{\Sbb}
\]
such that $\om_i=h_k-k_{k+1}$ holds for $k=i-\oh\in\Sbb$. Thus the negative discrete gradient of $\hb$ provides the system $\omb$. Reversing this gives the heights as a function of $\omb$:
\begin{equation}
h_k=\left\{
  \begin{aligned}
   &h_{\oh}-\sum_{i=1}^{k-\oh}\om_i&&\text{if }\oh<k\in\Sbb;\\
   &h_{\oh}+\sum_{i=k+\oh}^0\om_i&&\text{if }\oh>k\in\Sbb,
  \end{aligned}
\right.\label{eq:hdef}
\end{equation}
except for a constant simultaneous shift in all of the height values. We fix this integration constant by postulating $h_{\oh}(0)$ to be zero initially. Then, the dynamics \eqref{eq:dynwpq} translates to
\[
\hb\;\xrightarrow{\displaystyle p(\om_i,\om_{i+1})}\;\hb+\delta_{i+\oh},\qquad
\hb\;\xrightarrow{\displaystyle q(\om_i,\om_{i+1})}\;\hb-\delta_{i+\oh}.
\]
Equivalently, $h_{\oh}(t)-h_\oh(0)$ denotes the number of \emp{signed} particle hops that occurred above the bond $[0,\,1]$ until time $t\geq0$. In a similar fashion, one can think of $h_k(t)$ as the \emp{signed particle current} through the space-time line $(\oh,\,0)\to(k,t)$, where $k\in\Sbb$ and $t\in\Rbb^+$.
\begin{proof}[Proof of Theorem \ref{thm:scpdistr}]
Fix $t\geq0$, $n\in\Zbb$ and let $k=n+\oh$. We start two processes $\omb$ and $\etab$ with respective initial distributions $\sigb^{\vrho,\,\la}$ and $\tau_1\sigb^{\vrho,\,\la}$, where recall \eqref{eq:initprodmeas} and that
\begin{equation}\label{eq:shiftedprodinitmeas}
\tau_1\sigb^{\vrho,\,\la} = \bigotimes_{i=-\infty}^{-1}\nu^{\vrho}\otimes\bigotimes_{i=0}^{+\infty}\nu^{\la},
\end{equation}
which is just the left-shifted version of $\sigb^{\vrho,\,\la}$. The corresponding height functions are denoted by $\hb$ and $\gb$.

Closely following \cite{serf}, we will compute the following quantity in two different ways:
\begin{equation}
 \Exp_{\sigb^{\vrho,\,\la}} h_k(t)\; -\; \Exp_{\tau_1\sigb^{\vrho,\,\la}} g_k(t).\label{eq:hdiff}
\end{equation}
Define the \emp{possibly signed} site-marginal
\begin{align}\label{eq:nuhat}
\bar\nu^{\vrho,\,\la}(x,y) =
\begin{cases}
\nu^{\vrho}(\{z:z\leq y\})-\nu^{\la}(\{z:z\leq y-1\}) & \text{if $x=y$;}\\
\nu^{\la}(\{z:z\leq y\})-\nu^{\vrho}(\{z:z\leq y\}) & \text{if $x=y+1$;}\\
0 & \text{otherwise}
\end{cases}
\end{align}
on \(\Ic\times\Ic\), and the \emp{possibly signed} product measure
\[
 \bar{\mub}:\,=\bigotimes_{i=-\infty}^{-1}\nu^{\vrho,\,\vrho}\otimes\bar\nu^{\vrho,\,\la}\otimes\bigotimes_{i=1}^{+\infty}\nu^{\la,\,\la},
\]
with $\nu^{\vrho,\,\vrho}$, $\nu^{\la,\,\la}$ defined in \eqref{eq:diagonalmeasures}. Note that the marginal $\bar\nu^{\vrho,\,\la}$ is the one we have just elaborated in \eqref{eq:jointmeasure1}--\eqref{eq:jointmeasure2}. Proposition \ref{prop:probcouplingexist} tells that a coupling is possible with zero or one discrepancy if and only if $\bar\nu^{\vrho,\,\la}$ is a probability measure (see \eqref{eq:couplingstrict}). This does not need to be the case, $\bar\nu^{\vrho,\,\la}$, hence $\bar{\mub}$, might put negative mass for some initial configuration pairs. But thanks to Assumption \ref{assump:marginalmeasure}, it may only put negative mass on coinciding initial configurations, that is when $\omb(0)\equiv\etab(0)$. In other words, $\bar\nu^{\vrho,\,\la}(y,y)$ can possibly be negative but $\bar\nu^{\vrho,\,\la}(y+1,y)$ cannot.

The formal computations in the proof of Proposition \ref{prop:probcouplingexist} still work and show that $\bar\nu^{\vrho,\,\la}$ has respective first and second marginals $\nu^\vrho$ and $\nu^\la$. Therefore
\begin{multline*}
\Exp_{\sigb^{\vrho,\,\la}} h_k(t)-\Exp_{\tau_1\sigb^{\vrho,\,\la}} g_k(t)\\
\begin{aligned}
&=\sum_{x\in\Ic}\Exp_{\sigb^{\vrho,\,\la}}[h_k(t)\,|\,\om_0(0)=x]\cdot\nu^\vrho(x)-\sum_{y\in\Ic}\Exp_{\tau_1\sigb^{\vrho,\,\la}}[g_k(t)\,|\,\eta_0(0)=y]\cdot\nu^\vrho(y)\\
&=\sum_{x,y\in\Ic}\Bigl\{\Exp_{\sigb^{\vrho,\,\la}}[h_k(t)\,|\,\om_0(0)=x]-\Exp_{\tau_1\sigb^{\vrho,\,\la}}[g_k(t)\,|\,\eta_0(0)=y]\Bigr\}\cdot\bar\nu^{\vrho,\,\la}(x,y)\\
&=\sum_{y\in\Ic}\Bigl\{\Exp_{\sigb^{\vrho,\,\la}}[h_k(t)\,|\,\om_0(0)=y]-\Exp_{\tau_1\sigb^{\vrho,\,\la}}[g_k(t)\,|\,\eta_0(0)=y]\Bigr\}\cdot\bar\nu^{\vrho,\,\la}(y,y)\\
&\quad+\sum_{y\in\Ic}\Bigl\{\Exp_{\sigb^{\vrho,\,\la}}[h_k(t)\,|\,\om_0(0)=y+1]-\Exp_{\tau_1\sigb^{\vrho,\,\la}}[g_k(t)\,|\,\eta_0(0)=y]\Bigr\}\cdot\bar\nu^{\vrho,\,\la}(y+1,y).
\end{aligned}
\end{multline*}
In the first to last line the conditional initial distributions agree for $\omb$ and $\etab$, hence so do the expectations and we get zero. In the last line the mass $\bar\nu^{\vrho,\,\la}(y+1,y)$ is non-negative, and we replace this mass via \eqref{eq:scpcouplatorig}:
\begin{multline*}
\Exp_{\sigb^{\vrho,\,\la}} h_k(t)-\Exp_{\tau_1\sigb^{\vrho,\,\la}} g_k(t)\\
\begin{aligned}
&=\sum_{y\in\Ic}\Bigl\{\Exp_{\sigb^{\vrho,\,\la}}[h_k(t)\,|\,\om_0(0)=y+1]-\Exp_{\tau_1\sigb^{\vrho,\,\la}}[g_k(t)\,|\,\eta_0(0)=y]\Bigr\}\cdot\bar\nu^{\vrho,\,\la}(y+1,y)\\
&=\sum_{y\in\Ic}\Bigl\{\Exp_{\sigb^{\vrho,\,\la}}[h_k(t)\,|\,\om_0(0)=y+1]-\Exp_{\tau_1\sigb^{\vrho,\,\la}}[g_k(t)\,|\,\eta_0(0)=y]\Bigr\}\cdot\hat{\nu}^{\vrho,\,\la}(y+1,y)\cdot(\vrho-\la).
\end{aligned}
\end{multline*}
As $\hat{\nu}^{\vrho,\,\la}$ is a proper probability distribution, at this moment we can reunify the conditional expectations at the last display to obtain
\begin{multline*}
\Exp_{\sigb^{\vrho,\,\la}} h_k(t)-\Exp_{\tau_1\sigb^{\vrho,\,\la}}g_k(t)\\
\begin{aligned}
&=\sum_{y\in\Ic}\hat\Exp[\hat{h}_k(t)-\hat{g}_k(t)\,|\,\hat\om_0(0)=y+1,\,\hat\eta_0(0)=y]
\cdot\hat{\nu}^{\vrho,\,\la}(y+1,y)\cdot(\vrho-\la)\\
&=(\vrho-\la) \cdot \hat\Exp[\hat{h}_k(t)-\hat{g}_k(t)]
\end{aligned}
\end{multline*}
with $\hat\Exp$ denoting the measure with initial distribution $\hat{\mub}^{\vrho,\,\la}$ of \eqref{eq:scpstart} and following the basic coupling for evolution. Finally, notice that under the basic coupling with only one initial second class particle at the origin we have $\hat{h}_k(t)-\hat{g}_k(t)=\ind\{Q(t)>n\}$ a.s., thus
\begin{equation}\label{eq:hdiffc}
\Exp_{\sigb^{\vrho,\,\la}} h_k(t)-\Exp_{\tau_1\sigb^{\vrho,\,\la}} g_k(t)=(\vrho-\la)\cdot\hat\Prob\{Q(t)>n\}.
\end{equation}

Next, the second way of computing \eqref{eq:hdiff}. We notice that starting $\omb$ in distribution $\sigb^{\vrho,\,\la}$, letting it evolve according to its dynamics \eqref{eq:dynwpq}, and then defining
\begin{equation*}
\eta_i(s):\,=\om_{i+1}(s)\qquad \text{for all $i\in\Zbb$ and $0\leq s\leq t$}
\end{equation*}
gives another joint realization with an implied expectation and with the corresponding joint initial measure in which the marginal distributions of $\omb$ and $\etab$ are the same as before. Notice that the height functions in this coupling satisfy
\begin{align*}
g_k(t)
&=g_k(t)-g_{-\oh}(t)+g_{-\oh}(t)-g_{-\oh}(0)+g_{-\oh}(0)\\
&=\left\{
    \begin{aligned}
     &-\sum_{i=0}^{k-\oh}\eta_i(t),&&\text{if }k>-\oh,\\
     &0,&&\text{if }k=-\oh,\\
     &\sum_{i=k+\oh}^{-1}\eta_i(t),&&\text{if }k<-\oh
    \end{aligned}
   \right\}+g_{-\oh}(t)-g_{-\oh}(0)+g_{\oh}(0)+\eta_0(0)\\
&=\left\{
    \begin{aligned}
     &-\sum_{i=1}^{k+\oh}\om_i(t),&&\text{if }k>-\oh,\\
     &0,&&\text{if }k=-\oh,\\
     &\sum_{i=k+\thh}^0\om_i(t),&&\text{if }k<-\oh
    \end{aligned}
   \right\}+h_{\oh}(t)-h_{\oh}(0)+0+\om_1(0)\\
&=h_{k+1}(t)-h_{\oh}(t)+h_{\oh}(t)-h_{\oh}(0)+\om_1(0)\\
&=h_k(t)-\om_{k+\oh}(t)+\om_1(0),
\end{align*}
where we used \eqref{eq:hdef}, $h_\oh(0)=g_\oh(0)=0$, and the fact that the heights $g_{-\oh}$ and $h_\oh$ change at the same time under this coupling. Thus
\[
  \Exp_{\sigb^{\vrho,\,\la}} h_k(t) - \Exp_{\tau_1\sigb^{\vrho,\,\la}} g_k(t)= \Exp_{\sigb^{\vrho,\,\la}}\om_{n+1}(t)-\Exp_{\sigb^{\vrho,\,\la}}\om_1(0)=\Exp_{\sigb^{\vrho,\,\la}}\om_{n+1}(t)-\la.
\]
Together with \eqref{eq:hdiffc} we arrive to the desired identity
\[
\hat\Prob\{Q(t)>n\}=\frac{\Exp_{\sigb^{\vrho,\,\la}}\om_{n + 1}(t)-\la}{\vrho-\la},
\]
which finishes the proof.
\end{proof}
\begin{proof}[Proof of Theorem \ref{thm:scpbackground}]
We prove formula \eqref{eq:scpbackgroundformula} in a somewhat similar fashion as we did for \eqref{eq:findispldistr} before. But instead of using again the height functions we are going to work with the occupation numbers directly.
Let $\omb$ be a process starting from $\sigb^{\vrho,\,\la}$ and evolving according to \eqref{eq:infgen}. Then fix $t\geq 0$, recall \eqref{eq:initprodmeas} and that for some $\vphi:\Ic\to\Rbb$ and $x\in\Ic$: $\Phi(x)=\sum_{y=\omin}^x\vphi(y)$, and define
\[
A_n :\,= \Exp_{\sigb^{\vrho,\,\la}}[\Phi(\om_{n}(t))] - \Exp_{\sigb^{\vrho,\,\la}}[\Phi(\om_{n+1}(t))]
\]
for every $n\in\Zbb$. Now, by translation invariance we have
\[
A_n = \Exp_{\sigb^{\vrho,\,\la}}[\Phi(\hat\om_{n}(t))] - \Exp_{\tau_1\sigb^{\vrho,\,\la}}[\Phi(\hat\eta_{n}(t))]
\]
in which $\etab$ evolves according to the same dynamics \eqref{eq:infgen} as $\omb$ but starts from $\tau_1\sigb^{\vrho,\,\la}$ instead (see \eqref{eq:shiftedprodinitmeas}). For this latter quantity one can apply the same coupling technique we worked out in the previous proof. It then follows that
\begin{align*}
A_n
&=\sum_{x\in\Ic}\Exp_{\sigb^{\vrho,\,\la}}[\Phi(\om_{n}(t))\,|\,\om_0(0)=x]\cdot\nu^{\vrho}(x)-\sum_{y\in\Ic} \Exp_{\tau_1\sigb^{\vrho,\,\la}}[\Phi(\eta_{n}(t))\,|\,\eta_0(0)=y]\cdot\nu^{\la}(y)\\
&=\sum_{x,y\in\Ic}\Bigl\{\Exp_{\sigb^{\vrho,\,\la}}[\Phi(\om_{n}(t))\,|\,\om_0(0)=x] -  \Exp_{\tau_1\sigb^{\vrho,\,\la}}[\Phi(\eta_{n}(t))\,|\,\eta_0(0)=y]\Bigr\}\cdot\bar\nu^{\vrho,\,\la}(x,y)\\
&=\sum_{y\in\Ic}\Bigl\{\Exp_{\sigb^{\vrho,\,\la}}[\Phi(\om_{n}(t))\,|\,\om_0(0)=y] -  \Exp_{\tau_1\sigb^{\vrho,\,\la}}[\Phi(\eta_{n}(t))\,|\,\eta_0(0)=y]\Bigr\}\cdot\bar\nu^{\vrho,\,\la}(y,y)\\
&+\sum_{y\in\Ic}\Bigl\{\Exp_{\sigb^{\vrho,\,\la}}[\Phi(\om_{n}(t))\,|\,\om_0(0)=y+1] - \Exp_{\tau_1\sigb^{\vrho,\,\la}}[\Phi(\eta_{n}(t))\,|\,\eta_0(0)=y]\Bigr\}\cdot\bar\nu^{\vrho,\,\la}(y+1,y),
\end{align*}
where we used that the joint measure $\bar\nu^{\vrho,\,\la}$, defined in \eqref{eq:nuhat}, has respective marginals $\nu^{\vrho}$ and $\nu^{\la}$. As before we can take advantage of the fact that the last but one sum in the previous display vanishes and that $\bar\nu^{\vrho,\,\la}(y+1,y)=\hat\nu^{\vrho,\,\la}(y+1,y)\cdot(\vrho-\la)$ so
\begin{align*}
A_n
&=\sum_{y\in\Ic}\hat\Exp[\Phi(\hat\om_n(t))-\Phi(\hat\eta_n(t))\,|\,\hat\om_0(0)=y+1,\hat\eta_0(0)=y]\cdot\hat{\nu}^{\vrho,\,\la}(y+1,y)\cdot(\vrho-\la)\\
&=(\vrho-\la)\cdot\hat\Exp[\ind\{Q(t)=n\}\cdot\vphi(\hat\om_n(t))]
\end{align*}
using again that $\hat\nu$ is a proper probability distribution and
\[
\Phi(\hat\om_n(t))-\Phi(\hat\eta_n(t))=\ind\{Q(t)=n\}\cdot\vphi(\hat\om_n(t))
\]
a.s.\ under the basic coupling.

Thus we obtained the identity
$\hat\Exp[\ind\{Q(t)=n\}\cdot\vphi(\hat\om_{Q(t)}(t))] = \frac{1}{\vrho-\la}\cdot A_n$ for every $n\in\Zbb$. Now, summing these equations up as $n$ runs from $-L$ to $L$ for some $L\in\Zbb^+$, we arrive to
\begin{equation}\label{eq:summedupQ}
\hat\Exp[\ind\{Q(t)\in[-L,L]\}\cdot\vphi(\hat\om_{Q(t)}(t))] = \frac{1}{\vrho-\la}\cdot \big(\Exp_{\sigb^{\vrho,\,\la}}\Phi(\om_{-L}(t))-\Exp_{\sigb^{\vrho,\,\la}}\Phi(\om_{L+1}(t))\big).
\end{equation}
Notice that $\Exp_{\sigb^{\vrho,\,\la}}\Phi(\om_{\pm L}(t))$ equals to $\Exp_{\tau_{\pm L}\sigb^{\vrho,\,\la}}\Phi(\om_{0}(t))$ for all $L\in\Zbb^+$ due to the translation invariant property of the rates.

First, assume that $\vphi\geq0$ and $\Exp_{\sigb^{\vrho,\,\vrho}}\Phi(\om_0(t))<+\infty$ hold. Since the measures $\tau_{\pm L}\sigb^{\vrho,\,\la}$ are dominated by $\sigb^{\vrho,\,\vrho}$ for every $L\in\Nbb_0$ and $\Phi$ is monotone non-decreasing, it then follows by attractiveness that each of the above expectations of $\Phi(\om_0(t))$ is finite. Let $M$ be a positive real and define $\Phi_M$ to be $\Phi\wedge M$. It is easy to see that $\Phi_M$ and $\Phi-\Phi_M$ are non-negative and monotone non-decreasing functions as well. Hence
\begin{multline*}
\abs{\Exp_{\tau_{\pm L}\sigb^{\vrho,\,\la}}\Phi(\om_{0}(t)) - \Exp_{\sigb^{\vrho,\,\vrho}}\Phi(\om_0(t))}\\
\leq \abs{\Exp_{\tau_{\pm L}\sigb^{\vrho,\,\la}}\Phi_M(\om_{0}(t)) - \Exp_{\sigb^{\vrho,\,\vrho}}\Phi_M(\om_0(t))} +
2\cdot\Exp_{\sigb^{\vrho,\,\vrho}}\! (\Phi-\Phi_M)(\om_0(t)),
\end{multline*}
using again the attractiveness property of $\omb$. At this point $M$ can be made large enough for the last quantity to be small, independently of $L$. Also, note that the measures $\tau_{L}\sigb^{\vrho,\,\la}$ converge weakly to $\sigb^{\vrho,\,\vrho}$ and to $\sigb^{\la,\,\la}$ as $L\to\pm\infty$, respectively. Now, due to a finite speed of propagation of information (implied by construction of the dynamics), and taking advantage of that $\Phi_M$ is bounded, it follows that $\lim_{L\to+\infty}\Exp_{\tau_{L}\sigb^{\vrho,\,\la}}\Phi_M(\om_{0}(t))=\Exp_{\sigb^{\vrho,\,\vrho}}\Phi_M(\om_0(t))$ and $\lim_{L\to-\infty}\Exp_{\tau_{L}\sigb^{\vrho,\,\la}}\Phi_M(\om_{0}(t))=\Exp_{\sigb^{\la,\,\la}}\Phi_M(\om_0(t))$. What we have just proved implies that the right-hand side of \eqref{eq:summedupQ} is bounded hence the left-hand side converges to $\hat\Exp\,\vphi\big(\hat\om_{Q(t)}(t)\big)$ by monotone convergence, and so we get the desired formula \eqref{eq:scpbackgroundformula}.

On the other hand, assume that $\vphi$ is absolutely summable: $\sum_{y\in\Ic}\abs{\vphi(y)}<+\infty$. It follows that $\vphi$ as well as $\Phi$ are bounded functions. By dominated convergence the left-hand side of \eqref{eq:summedupQ} converges to $\hat\Exp\,\vphi\big(\hat\om_{Q(t)}(t)\big)$ as $L\to+\infty$, while the convergence of the right-hand side follows directly from finite information propagation velocity and the weak limit of $\tau_{\pm L}\sigb^{\vrho,\,\la}$ as $L\to+\infty$.

We remark that the way we obtained \eqref{eq:scpbackgroundformula} could have been used to obtain \eqref{eq:findispldistr} as well. For historical reasons in the previous case we followed the approach that was inherited from \cite{serf}.
\end{proof}
\begin{proof}[Proof of Proposition \ref{prop:landimlocconv}]
Basically, we will follow the lines of \cite[pp. 1791--1792]{landim}. Fix a $t\in\Rbb^+$ and let $x\in\Rbb$ be a continuity point of $u^{0}(\,\cdot\,,t)$. Furthermore, let $\vphi:\Omega\to\Rbb$ be a monotone non-decreasing cylinder function for which \eqref{eq:landimexpvlimit} holds.
Note that the measure $\tau_{j}\sigb^{\vrho,\,\la}$ is stochastically dominated by $\tau_{k}\sigb^{\vrho,\,\la}$ for every $j\geq k$. Then by attractiveness we have
\[
 \Exp_{\tau_{[N\veps]}\sigb^{\vrho,\,\la}}\vphi(\tau_k\omb(N t))\leq
 \Exp_{\sigb^{\vrho,\,\la}}\vphi\big(\tau_{[Nx]}\omb(N t)\big)\leq \Exp_{\tau_{[N\veps]}\sigb^{\vrho,\,\la}}\vphi(\tau_j\omb(N t))
\]
for every $j \geq [Nx] - [N\veps] \geq k$, where $\veps\in\Rbb$. Now, let $\veps_+>0$ and $\veps_{-}<0$. Then
\begin{multline*}
\frac{1}{\abs{[N\veps_{-}]}+1}\sum_{k\,:\,[Nx]-\abs{[N\veps_{-}]}\leq k\leq [Nx]}\Exp_{\tau_{[N\veps_{-}]}\sigb^{\vrho,\,\la}}\vphi(\tau_j\omb(N t))\\
\leq \Exp_{\sigb^{\vrho,\,\la}}\vphi\big(\tau_{[N x]}\omb(N t)\big)
\leq \frac{1}{[N\veps_{+}]+1} \sum_{j\,:\,[Nx]\leq j \leq [Nx] + [N\veps_{+}]} \Exp_{\tau_{[N\veps_{+}]}\sigb^{\vrho,\,\la}}\vphi(\tau_j\omb(N t)).
\end{multline*}
Taking the limit inferior then the superior as $N\to+\infty$ and using assumption \eqref{eq:landimexpvlimit} we obtain that
\begin{multline*}
\frac{1}{\abs{\veps_{-}}}\int_{z\,:\,x-\abs{\veps_{-}}\leq z\leq x}\Exp_{\pib^{u^{\veps_{-}}(z,t)}}\vphi(\omb)\di z\\
\leq
\liminf_{N\to+\infty}\Exp_{\sigb^{\vrho,\,\la}}\vphi\big(\tau_{[N x]}\omb(N t)\big)
\leq
\limsup_{N\to+\infty}\Exp_{\sigb^{\vrho,\,\la}}\vphi\big(\tau_{[N x]}\omb(N t)\big)\\
\leq
\frac{1}{\veps_{+}}\int_{y\,:\,x\leq y\leq x+\veps_{+}}\Exp_{\pib^{u^{\veps_{+}}(y,t)}}\vphi(\omb)\di y.
\end{multline*}
Since $(\pib^{\vrho})_{\vrho\in\Rbb}$ is an ordered family of measures and $u^{\veps_{-}}$ as well as $u^{\veps_{+}}$ are monotone non-increasing functions it follows that
\begin{multline*}
\Exp_{\pib^{u^{\veps_{-}}(x,t)}}\vphi(\omb)
\leq \liminf_{N\to+\infty}\Exp_{\sigb^{\vrho,\,\la}}\vphi\big(\tau_{[N x]}\omb(N t)\big)
\leq \limsup_{N\to+\infty}\Exp_{\sigb^{\vrho,\,\la}}\vphi\big(\tau_{[N x]}\omb(N t)\big)\\
\leq \Exp_{\pib^{u^{\veps_{+}}(x,t)}}\vphi(\omb).
\end{multline*}
Without loss of generality we can assume that $\vphi\geq0$. Now, let $M\in\Rbb^+$ and define $\vphi_M=\vphi\wedge M$. Then
\begin{multline*}
\abs{\Exp_{\pib^{u^{\veps_{\pm}}(x,t)}}\vphi(\omb)-\Exp_{\pib^{u^{0}(x,t)}}\vphi(\omb)}\\
\leq \abs{\Exp_{\pib^{u^{\veps_{\pm}}(x,t)}}\vphi_M(\omb)-\Exp_{\pib^{u^{0}(x,t)}}\vphi_M(\omb)}
+ 2\Exp_{\pib^{\vrho_{\max}}}\big(\vphi-\vphi_M\big)(\omb)
\end{multline*}
by stochastic dominance of the measures $(\pib^{\vrho})_{\vrho\in\Rc}$. Note that the last quantity in the previous display can be made arbitrarily small by choosing $M$ to be large enough while the first one vanishes as $\veps\to0$ by weak convergence. This results in the desired limit \eqref{eq:stronghdl}.

Since any bounded cylinder function can be written as the difference of two monotone cylinder functions we have finished the proof.
\end{proof}
Now, we can turn to the proofs of the main results (see Section \ref{sec:mainr}).
\begin{proof}[Proof of Theorem \ref{thm:scphydroasym}]
First, assume that our attractive process has bounded occupation numbers. Hence we can apply Theorem \ref{thm:bagurasa}. Let us mention that the density parametrized ergodic (or extremal) set of stationary distributions have been proved to exist and form a continuous and stochastically ordered family of measures with a closed parameter set $\Rc$ of $[\omin, \omax]$ containing $\omin$ and $\omax$ (see \cite[Section 3, Proposition 3.1]{bagurasa}). So for applying Proposition \ref{prop:landimlocconv} we only need to take care of the limit \eqref{eq:landimexpvlimit} when $\vphi(\omb)=\om_0$. Without loss of generality one can assume that $\omin\geq 0$ and that the compact support of $\psi$ has (at most) unit Lebesgue measure. Then we have
\begin{multline*}
\Exp_{\tau_{[N\veps]}\sigb^{\vrho,\,\la}}\bigg|\frac{1}{N}\sum_{j\in\Zbb}\psi(j/N)\cdot \om^{\veps,N}_j(N t)-\int_{x\in\Rbb}\psi(x)\cdot u^{\veps}(x,t) \di x\bigg|\\
\leq \veps + \omax\cdot\max_{x\in\Rbb}\abs{\psi(x)}
\cdot\Prob^N\bigg(\bigg|\frac{1}{N}\sum_{j\in\Zbb}\psi(j/N)\cdot \om^{\veps,N}_j(N t)-\int_{x\in\Rbb}\psi(x)\cdot u^{\veps}(x,t) \di x\bigg|>\veps\bigg),
\end{multline*}
where $\veps>0$ can be chosen arbitrarily. Now, taking the limit as $N\to+\infty$ we get \eqref{eq:landimexpvlimit}. Finally, we can conclude using \eqref{eq:findispldistr}.

On the other hand, considering one of the processes of the misanthrope family with bounded rates one can apply Theorem \ref{thm:rezakhydro}. So we have in hand both the set of product-form extremal stationary distributions and the convergence result \eqref{eq:rezakexpvlimit} (see also \cite[Section 7]{hl}). Hence, we can directly apply Proposition \ref{prop:landimlocconv} with $\vphi(\omb) = \om_0$ and the desired limit \eqref{eq:scpspeedlim} can be obtained via \eqref{eq:findispldistr} again.

In both cases the monotonicity, boundedness and convergence of the entropy solutions $(u^{\veps})_{\veps\in\Rbb}$ come from the classical results of hyperbolic conservation laws (see \cite{holdenrisebro} and further references therein).
\end{proof}
\begin{proof}[Proof of Theorem \ref{thm:scpdiffusivelim}]
By shifting the whole system upwards we can assume, without loss of generality, that $\omin=0$. Now, we are going to handle both the finite and infinite settings. Attractiveness and the fact $\om_n(t)\geq0$ imply that $\sup_{n\in\Zbb}\Exp_{\sigb^{\vrho,\,\la}}\om_n(t)^2$ can be estimated from above by $\Exp_{\sigb^{\vrho,\,\vrho}}\omb_0(t)^2$ which is supposed to be finite for every fixed $t\geq0$. It follows that for each $t\geq0$, the sequence
\[
\bigg(\frac{1}{N}\sum_{j\in\Zbb}\psi(j/N)\cdot \om^{\veps,N}_j(N^2 t)\bigg)_{N\in\Zbb^+}
\]
is bounded in $\Lc^2$ for every $\veps\in\Rbb$, where $\omb^{\veps,N}$ starts from $\tau_{[N\veps]}\sigb^{\vrho,\,\la}$ and $\psi$ is a given continuous function of compact support. Hence the conservation of local equilibrium of gradient processes implies \eqref{eq:landimexpvlimit} for $\vphi(\omb)=\om_0$. Since the time scaling played no role in (the proof of) Proposition \ref{prop:landimlocconv}, one can save this result to the diffusive case as well, resulting in the desired convergence \eqref{eq:scpdiffusivelim}.
\end{proof}
\begin{proof}[Proof of Theorem \ref{thm:collprob}]
Without loss of generality we can assume that $\omin=0$ while $\omax$ is some fixed positive integer.
We start the process $\hat\etab$ from initial configuration $\hat\etab(0)=\omax\ind\{i\leq 0\}$. For sake of brevity we let $\bar r:\,= p(\hat\eta_0(0),\hat\eta_1(0))$, which is positive by the non-degeneracy of the rates. In what follows, we will work with the height function $\gb$ of $\hat\etab$ (see its definition in \eqref{eq:hdef}).

Since the initial configuration $\hat\etab_0$ of $\hat\etab$ can change only due to a jump performed by a particle from the origin to lattice point $1$, it follows that in a small time interval $[0,\veps]$ we should only take into account two events: $\hat\etab(\veps)=\hat\etab_0$ occurring with probability $1-\veps\cdot \bar{r}$ (up to first order in $\veps$) and $\hat\etab(\veps)=\hat\etab_0-\delta_0+\delta_1$ which in turn occurs with rate $\veps\cdot \bar{r}$. All the other moves are of order $\mathrm{o}(\veps)$ in probability. Putting the above together we arrive to
\begin{align}
\frac{\di}{\di t}\Exp_{\hat\etab_0} g_{\oh}(t)
&= \lim_{\veps\downarrow0}\frac{\Exp_{\hat\etab_0} g_{\oh}(t+\veps)-\Exp_{\hat\etab_0} g_{\oh}(t)}{\veps}\nonumber\\
&= \bar{r} \cdot \lim_{\veps\downarrow0}\bigg(\Exp_{\hat\etab_0}
\big\{\Exp [g_{\oh}(t+\veps)\,|\,\hat\etab(\veps)=\hat\etab_0-\delta_0+\delta_1]- \Exp[g_{\oh}(t+\veps)\,|\,\hat\etab(\veps)=\hat\etab_0]\big\}\bigg)\nonumber\\
&= \bar{r} \cdot \bigg(1 + \Exp_{\hat\omb_0}h_{\oh}(t)-\Exp_{\hat\etab_0}g_{\oh}(t)\bigg)\label{eq:kolmforw}
\end{align}
by the Markov property, where $\hb$ is the height function of $\hat\omb$ that starts from $\hat\omb_0=\hat\etab_0-\delta_0+\delta_1$. Along the way we have also used the fact that $g_{\oh}(t)$ counts exactly how many (signed) particle jumps occurred above the bond $[0,1]$ until time $t\geq0$. Also recall that $g_{\oh}(0)$ (and $h_{\oh}(0)$) is set to be zero by choice. (\eqref{eq:kolmforw} is also known as the \emp{Kolmogorov forward equation}.)

We now couple the processes $\hat\omb$ and $\hat\etab$ coordinate-wise, employing the basic coupling with deterministic initial configurations $\hat\omb_0$ and $\hat\etab_0$, respectively. Notice that there are two second class particles in the system $(\hat\omb,\hat\etab)$: a negative and a positive starting from positions $0$ and $1$, respectively. It then follows that under this coupling
\begin{equation}
1 + h_{\oh}(t) - g_{\oh}(t) = \sum_{j=1}^{+\infty} \big(\hat\om_j(t)-\hat\eta_j(t)\big) = \sum_{j=1}^{+\infty} s_j(t)\cdot n_j(t)\leq \ind\{\Nc(s)=2\text{ for all }0\leq s\leq t\}\label{eq:twoscpestim}
\end{equation}
holds a.s. On the other hand
\begin{align}
 \frac{\di}{\di t}\Exp_{\hat\etab_0} g_{\oh}(t)= \lim_{\veps\downarrow0}\frac{\Exp_{\hat\etab_0}\big[g_{\oh}(t+\veps) - g_{\oh}(t)\big]}{\veps}&= \lim_{\veps\downarrow0}\frac{\Exp_{\hat\etab_0}\big[\veps\cdot p(\hat\eta_0(t),\hat\eta_1(t)) - \veps\cdot q(\hat\eta_0(t),\hat\eta_1(t))\big]}{\veps}\nonumber\\
&= \Exp_{\hat\etab_0}[p(\hat\eta_0(t),\hat\eta_1(t))-q(\hat\eta_0(t),\hat\eta_1(t))],\label{eq:kolmbackw}
\end{align}
using again the Markov property and that $g_{\oh}$ can change only if a particle attempts to jump either from $0$ to $1$ or from $1$ to $0$. (\eqref{eq:kolmbackw} is sometimes called the \emp{Kolmogorov backward equation}.)

Putting \eqref{eq:kolmforw} and \eqref{eq:kolmbackw} together and using the estimate \eqref{eq:twoscpestim} we arrive to
\[
\frac{1}{\bar{r}}\Exp_{\hat\etab_0}[p(\hat\eta_0(t),\hat\eta_1(t))-q(\hat\eta_0(t),\hat\eta_1(t))] \leq \hat\Prob\big\{\Nc(s)=2\text{ for all }0\leq s\leq t\big\}.
\]
Now, taking the limit superior as $t\to+\infty$ we obtain the desired inequality \eqref{eq:collprobestimate} by monotone convergence.

In the totally asymmetric case, it is easy to see that there must exist a subsequence $(t_{m})_{m=1}^{+\infty}$ for which $\lim\limits_{m\to+\infty}\Exp_{\hat\etab_0}p(\hat\eta_0(t_m),\hat\eta_1(t_m))>0$,
since otherwise we would have $\lim\limits_{t\to+\infty} p(\hat\eta_0(t),\hat\eta_1(t))=0$ a.s.\ by dominated convergence, implying that the probability of the event $\{\hat\eta_0(t)=0 \text{ or } \hat\eta_1(t)=\omax\}$ tends to $1$ as $t\to+\infty$. But this obviously cannot happen.

Finally, for a misanthrope process (i.e.\ the rates of which satisfy the conditions of Theorem \ref{thm:prodstatergod}) we can apply Theorem \ref{thm:rezakhydro} and then Proposition \ref{prop:landimlocconv}. The proof is then completed by choosing the cylinder function $\vphi(\omb)$ to be $p(\om_0,\om_1)-q(\om_0,\om_1)$ ($\om\in\Omega$) in \eqref{eq:rezakexpvlimit}.
\end{proof}

\addcontentsline{toc}{section}{References}%
\bibliographystyle{plain}%
\bibliography{refsmarton}%

\begin{thebibliography}{10}

\bibitem{amirangel}
G.~Amir, O.~Angel, and B.~Valk{\'o}.
\newblock The {TASEP} speed process.
\newblock {\em Ann. Probab.}, 39(4):1205--1242, 2011.

\bibitem{and}
E.D. Andjel.
\newblock Invariant measures for the zero range processes.
\newblock {\em Ann. Probab.}, 10(3):525--547, 1982.

\bibitem{bagurasa}
C.~Bahadoran, H.~Guiol, K.~Ravishankar, and E.~Saada.
\newblock Euler hydrodynamics of one-dimensional attractive particle systems.
\newblock {\em Ann. Probab.}, 34(4):1339--1369, 2006.

\bibitem{bagurasastrong}
C.~Bahadoran, H.~Guiol, K.~Ravishankar, and E.~Saada.
\newblock Strong hydrodynamic limit for attractive particle systems on
  {$\mathbb Z$}.
\newblock {\em Electron. J. Probab.}, 15:no. 1, 1--43, 2010.

\bibitem{valak}
M.~Bal{\'a}zs.
\newblock Microscopic shape of shocks in a domain growth model.
\newblock {\em J. Statist. Phys.}, 105(3-4):511--524, 2001.

\bibitem{fluct}
M.~Bal{\'a}zs.
\newblock Growth fluctuations in a class of deposition models.
\newblock {\em Ann. Inst. H. Poincar\'e Probab. Statist.}, 39(4):639--685,
  2003.

\bibitem{sokvalak}
M.~Bal{\'a}zs.
\newblock Multiple shocks in bricklayers' model.
\newblock {\em J. Statist. Phys.}, 117(1-2):77--98, 2004.

\bibitem{rwshscp}
M.~Bal{\'a}zs, Gy. Farkas, P.~Kov{\'a}cs, and A.~R{\'a}kos.
\newblock Random walk of second class particles in product shock measures.
\newblock {\em J. Stat. Phys.}, 139(2):252--279, 2010.

\bibitem{b_n_t_t_coex16}
M.~Bal{\'a}zs, A.L. Nagy, B.~T{\'o}th, and I.~T{\'o}th.
\newblock Coexistence of shocks and rarefaction fans: Complex phase diagram of
  a simple hyperbolic particle system.
\newblock {\em Journal of Statistical Physics}, 165(1):115--125, 2016.

\bibitem{exists}
M.~Bal{\'a}zs, F.~Rassoul-Agha, T.~Sepp{\"a}l{\"a}inen, and S.~Sethuraman.
\newblock Existence of the zero range process and a deposition model with
  superlinear growth rates.
\newblock {\em Ann. Probab.}, 35(4):1201--1249, 2007.

\bibitem{convex}
M.~Bal{\'a}zs and T.~Sepp{\"a}l{\"a}inen.
\newblock A convexity property of expectations under exponential weights.
\newblock {\em \url{http://arxiv.org/abs/0707.4273}}, 2007.

\bibitem{varj2nd}
M.~Bal{\'a}zs and T.~Sepp{\"a}l{\"a}inen.
\newblock Exact connections between current fluctuations and the second class
  particle in a class of deposition models.
\newblock {\em J. Stat. Phys.}, 127(2):431--455, 2007.

\bibitem{catdob}
E.~Cator and S.~Dobrynin.
\newblock Behavior of a second class particle in hammersley's process.
\newblock {\em Electron. J. Probab.}, 11(26):670--685, 2006.

\bibitem{coco}
C.~Cocozza-Thivent.
\newblock Processus des misanthropes.
\newblock {\em Z. Wahrsch. Verw. Gebiete}, 70(4):509--523, 1985.

\bibitem{dls}
B.~Derrida, J.L. Lebowitz, and E.R. Speer.
\newblock Shock profiles for the asymmetric simple exclusion process in one
  dimension.
\newblock {\em J. Statist. Phys.}, 89(1-2):135--167, 1997.
\newblock Dedicated to Bernard Jancovici.

\bibitem{ferrarishocks}
P.A. Ferrari.
\newblock Shocks in one-dimensional processes with drift.
\newblock In {\em Probability and phase transition ({C}ambridge, 1993)}, volume
  420 of {\em NATO Adv. Sci. Inst. Ser. C Math. Phys. Sci.}, pages 35--48.
  Kluwer Acad. Publ., Dordrecht, 1994.

\bibitem{fgmcollision}
P.A. Ferrari, P.~Gon{\c{c}}alves, and J.B. Martin.
\newblock Collision probabilities in the rarefaction fan of asymmetric
  exclusion processes.
\newblock {\em Ann. Inst. Henri Poincar\'e Probab. Stat.}, 45(4):1048--1064,
  2009.

\bibitem{serf}
P.A. Ferrari and C.~Kipnis.
\newblock Second class particles in the rarefaction fan.
\newblock {\em Ann. Inst. H. Poincar\'e Probab. Statist.}, 31(1):143--154,
  1995.

\bibitem{ferrmarpiment}
P.A. Ferrari, J.B. Martin, and L.P.R. Pimentel.
\newblock A phase transition for competition interfaces.
\newblock {\em Ann. Appl. Probab.}, 19(1):281--317, 2009.

\bibitem{compint}
P.A. Ferrari and L.P.R. Pimentel.
\newblock Competition interfaces and second class particles.
\newblock {\em Ann. Probab.}, 33(4):1235--1254, 2005.

\bibitem{ferrpresuttivaressymzr}
P.A. Ferrari, E.~Presutti, and M.E. Vares.
\newblock Local equilibrium for a one-dimensional zero range process.
\newblock {\em Stochastic Process. Appl.}, 26(1):31--45, 1987.

\bibitem{gon_zr_raref}
P.~Gon{\c{c}}alves.
\newblock On the asymmetric zero-range in the rarefaction fan.
\newblock {\em J. Stat. Phys.}, 154(4):1074--1095, 2014.

\bibitem{guiolmountfordquestions}
H.~Guiol and T.~Mountford.
\newblock Questions for second class particles in exclusion processes.
\newblock {\em Markov Process. Related Fields}, 12(2):301--308, 2006.

\bibitem{holdenrisebro}
H.~Holden and N.H. Risebro.
\newblock {\em Front tracking for hyperbolic conservation laws}, volume 152 of
  {\em Applied Mathematical Sciences}.
\newblock Springer, New York, 2011.
\newblock First softcover corrected printing of the 2002 original.

\bibitem{cl}
C.~Kipnis and C.~Landim.
\newblock {\em Scaling limits of interacting particle systems}, volume 320 of
  {\em Grundlehren der Mathematischen Wissenschaften [Fundamental Principles of
  Mathematical Sciences]}.
\newblock Springer-Verlag, Berlin, 1999.

\bibitem{landim}
C.~Landim.
\newblock Conservation of local equilibrium for attractive particle systems on
  {${\bf Z}^d$}.
\newblock {\em Ann. Probab.}, 21(4):1782--1808, 1993.

\bibitem{landiminfinite}
C.~Landim and M.~Mourragui.
\newblock Hydrodynamic limit of mean zero asymmetric zero range processes in
  infinite volume.
\newblock {\em Ann. Inst. H. Poincar\'e Probab. Statist.}, 33(1):65--82, 1997.

\bibitem{ips}
T.M. Liggett.
\newblock {\em Interacting particle systems}, volume 276 of {\em Grundlehren
  der Mathematischen Wissenschaften [Fundamental Principles of Mathematical
  Sciences]}.
\newblock Springer-Verlag, New York, 1985.

\bibitem{mountguiol}
T.~Mountford and H.~Guiol.
\newblock The motion of a second class particle for the {TASEP} starting from a
  decreasing shock profile.
\newblock {\em Ann. Appl. Probab.}, 15(2):1227--1259, 2005.

\bibitem{perrut}
A.~Perrut.
\newblock Hydrodynamic limit for a nongradient system in infinite volume.
\newblock {\em Stochastic Process. Appl.}, 84(2):227--253, 1999.

\bibitem{hl}
F.~Rezakhanlou.
\newblock Hydrodynamic limit for attractive particle systems on {${\bf Z}\sp
  d$}.
\newblock {\em Comm. Math. Phys.}, 140(3):417--448, 1991.

\bibitem{rez}
F.~Rezakhanlou.
\newblock Microscopic structure of shocks in one conservation laws.
\newblock {\em Ann. Inst. H. Poincar\'e Anal. Non Lin\'eaire}, 12(2):119--153,
  1995.

\bibitem{romik}
D.~Romik and P.~{\'S}niady.
\newblock Jeu de taquin dynamics on infinite {Y}oung tableaux and second class
  particles.
\newblock {\em Ann. Probab.}, 43(2):682--737, 2015.

\bibitem{hkl}
T.~Sepp{\"a}l{\"a}inen.
\newblock Existence of hydrodynamics for the totally asymmetric simple
  {$K$}-exclusion process.
\newblock {\em Ann. Probab.}, 27(1):361--415, 1999.

\bibitem{transinvexcl}
T.~Sepp{\"a}l{\"a}inen.
\newblock Translation {I}nvariant {E}xclusion {P}rocesses.
\newblock Unpublished book, available at
  \url{http://www.math.wisc.edu/~seppalai/excl-book/ajo.pdf}, 2008.

\bibitem{hydro}
B.~T{\'o}th and B.~Valk{\'o}.
\newblock Between equilibrium fluctuations and {E}ulerian scaling: perturbation
  of equilibrium for a class of deposition models.
\newblock {\em J. Statist. Phys.}, 109(1-2):177--205, 2002.

\bibitem{tracywidom}
C.~A. Tracy and H.~Widom.
\newblock On the distribution of a second-class particle in the asymmetric
  simple exclusion process.
\newblock {\em J. Phys. A}, 42(42):425002, 6, 2009.

\end{thebibliography}

\end{document}